\documentclass[12pt]{amsart}

\usepackage[pdfauthor   = {Mohamed\ Barakat\ and\ Markus\ Lange-Hegermann},
            pdftitle    = {An\ algorithmic\ approach\ to Chevalley's\ Theorem\ on\ images\ of\ rational\ morphisms\ between\ affine\ varieties},
            pdfsubject  = {},
            pdfkeywords = {Constructible\ images;\ Chevalley's\ Theorem;\ Gröbner\ bases;\ elimination},
            bookmarks=true,
            bookmarksopen=true,
            pagebackref=true,
            hyperindex=true,
            colorlinks=true,
            linkcolor=blue,
            citecolor=blue,
            filecolor=blue,
            urlcolor=blue,
            ]{hyperref}

\usepackage[utf8,utf8x]{inputenc}
\usepackage[english]{babel}
\usepackage[T1]{fontenc}
\usepackage{geometry}                
\usepackage{microtype}

\usepackage{times}
\usepackage{mathrsfs}
\usepackage{latexsym}
\usepackage{amsmath}
\usepackage{amssymb}
\usepackage{amsthm}
\usepackage{mathtools} 
\usepackage{extpfeil} 
\usepackage[shortlabels]{enumitem} 
\usepackage{tikz}
\usetikzlibrary{shapes,arrows,matrix,backgrounds,positioning,plotmarks,calc,patterns,
decorations.shapes,
decorations.fractals,
decorations.markings,
decorations.pathreplacing,
decorations.pathmorphing,
decorations.text
}
\usepackage[colorinlistoftodos,shadow]{todonotes}

\usepackage[active]{srcltx} 

\usepackage[sort&compress,capitalise]{cleveref}
\crefname{subsection}{subsection}{subsections}

\usepackage[linesnumbered,commentsnumbered,ruled,vlined]{algorithm2e}

\SetCommentSty{mycommfont}
\SetKw{Continue}{continue}

\newtheoremstyle{mytheoremstyle} 
    {5pt}                    
    {5pt}                    
    {\itshape}                   
    {\parindent}                           
    {\bf}                   
    {.}                          
    {.5em}                       
    {}  

\theoremstyle{mytheoremstyle}

\newtheorem{theorem}{Theorem}[section]

\newtheorem{prop}[theorem]{Proposition}
\newtheorem{coro}[theorem]{Corollary}

\newtheoremstyle{mytdefintionstyle} 
    {5pt}                    
    {5pt}                    
    {\rm}                   
    {\parindent}                           
    {\bf}                   
    {.}                          
    {.5em}                       
    {}  

\theoremstyle{remark}
\newtheorem{rmrk}[theorem]{Remark}

\theoremstyle{mytdefintionstyle}
\newtheorem{defn}[theorem]{Definition}
\newtheorem{exmp}[theorem]{Example}

\newtheoremstyle{exmp_contd}
    {5pt}                    
    {5pt}                    
    {\rm}                   
    {\parindent}                           
    {\bf}                   
    {.}                          
    {.5em}                       
    {\thmname{#1}\ \thmnumber{ #2}\thmnote{#3}\ (continued)}  
\theoremstyle{exmp_contd}

\newcommand\nameft\textrm

\newcommand{\rbh}{{\mathrm{rbh}}}

\DeclareMathOperator{\img}{im}

\DeclareMathOperator{\Hom}{Hom}

\DeclareMathOperator{\Spec}{Spec}

\newcommand\p{\mathfrak{p}}

\newcommand\OO{\mathcal{O}}
\newcommand\PP{\mathbb{P}}
\renewcommand\AA{\mathbb{A}}

\newcommand\A{\mathcal{A}}

\newcommand{\D}{\mathcal{D}}

\newcommand\mcI{\mathcal{I}}

\newcommand{\Q}{\mathbb{Q}}
\newcommand{\R}{\mathbb{R}}

\newcommand{\V}{\mathcal{V}}

\newcommand{\Z}{\mathbb{Z}}

\renewcommand\phi{\varphi}

\definecolor{darkgray}{rgb}{0.3,0.3,0.3}

\pgfarrowsdeclarecombine[\pgflinewidth]
  {doublestealth}{doublestealth}{stealth'}{stealth'}{stealth'}{stealth'}

\definecolor{darkgreen}{rgb}{0.008,0.617,0.067}
\definecolor{brown}{rgb}{0.6,0.4,0.2}

\author{Mohamed Barakat}
\address{Department of mathematics, University of Siegen, 57068 Siegen, Germany}
\email{\href{mailto:Mohamed Barakat <mohamed.barakat@uni-siegen.de>}{mohamed.barakat@uni-siegen.de}}

\author{Markus Lange-Hegermann}
\address{Department of electrical engineering and computer science, Ostwestfalen-Lippe University of Applied Sciences and Arts, 32657 Lemgo, Germany}
\email{\href{mailto:Markus Lange-Hegermann <markus.lange-hegermann@th-owl.de>}{markus.lange-hegermann@th-owl.de}}

\begin{document}

\title[An algorithmic approach to Chevalley's Theorem]{An algorithmic approach to Chevalley's Theorem on images of rational morphisms between affine varieties}

\begin{abstract}
  The goal of this paper is to introduce a new constructive geometric proof of the affine version of Chevalley's Theorem.
  This proof is algorithmic and a verbatim implementation resulted in an efficient code for computing the constructible image of rational maps between affine varieties.
  Our approach extends the known descriptions of uniform matrix product states to $\operatorname{uMPS}(2,2,5)$.
\end{abstract}

\thanks{This work is a contribution to Project II.1 of SFB-TRR 195 'Symbolic Tools in Mathematics and their Application' funded by Deutsche Forschungsgemeinschaft (DFG)}

\maketitle


\section{Introduction}\label{section_introduction}

\setcounter{theorem}{-1}
\begin{theorem}[Chevalley] \label{thm:Chevalley_general}
  Let $f: X \to Y$ be rational map of affine varieties and $C$ a constructible subset of $X$.
  Then the image\footnote{By ``image'' we always mean the set-theoretic image. The scheme-theoretic image is closed by definition; it is the vanishing locus of the elimination ideal of the vanishing ideal of the graph of $f$.} $f(C)$ is again constructible.
\end{theorem}

This paper develops a geometric proof of the above affine version of Chevalley's Theorem which is simultaneously a description of an efficient algorithm for computing the constructible image $f(C)$.
The general case of Chevalley's Theorem reduces to the affine case.

The central idea of all geometric and algebraic proofs of Chevalley's Theorem we are aware of is the use of the following \Cref{thm:lca_of_image} which ensures the existence of a \emph{proper} closed subset $D$ of $\overline{f(C)}$ containing the (constructible) set $\overline{f(C)} \setminus f(C)$ of all points outside of the image:
\begin{equation*}
  \overline{f(C)} \setminus f(C) \subseteq D \subsetneq \overline{f(C)} \mbox{.}
\end{equation*}
We call such a subset $D$ a \emph{relative boundary hull} of $f(C) \subseteq Y$ (cf.~\Cref{section_constructible}).
\begin{theorem}\label{thm:lca_of_image}
  Let $f: X \to Y$ be a rational map of affine varieties and $C$ a Zariski-closed subset of $X$.
  Then the image $f(C)$ admits a relative boundary hull $D$.
\end{theorem}

This already suffices to prove the following special form of Chevalley's Theorem, which in turn easily implies \Cref{thm:Chevalley_general} by a generalized version of Rabinowitsch trick, as shown in \Cref{proof_Chevalley_general}.
\begin{theorem}[Chevalley] \label{thm:Chevalley}
  Let $f: X \to Y$ be a rational map of affine varieties and $C$ a Zariski-closed subset of $X$.
  Then the image $f(C)$ is constructible.
\end{theorem}

All proofs of \Cref{thm:Chevalley} basically use the same iterative algorithm and only differ in their choice of the relative boundary hull $D$ of $f(C)$ constructed in the proof of \Cref{thm:lca_of_image}.
Indeed, the choice of $D$, in particular the number of its irreducible components and their degrees have a profound impact on the computational efficiency of the resulting iteration.

The main reduction step in this iterative algorithm relies on the properness of $D$ as a subset of $\overline{f(C)}$.
More precisely, the projection formula $f( f^{-1}(D) \cap C ) = D \cap f(C)$ yields the equation
\begin{align*} \label{eq:sum} \tag{$\uplus$}
  f(C) &= (\overline{f(C)} \setminus D) \uplus (D \cap f(C))
  \\
  &= (\underbrace{\overline{f(C)} \setminus D}_{\text{locally closed}}) \uplus f( f^{-1}(D) \cap C ) \mbox{,}
  \nonumber
\end{align*}
which decomposes $f(C)$ into the (disjoint) union of the \emph{locally closed} subset $\overline{f(C)} \setminus D$ and the image under $f$ of the strictly smaller subset $f^{-1}(D) \cap C \subsetneq C$ of points of $C$ lying over the relative boundary hull $D$.
The termination of the iteration is guaranteed by the Noetherianity of $X$ or of $Y$.
This proves \Cref{thm:Chevalley} as a corollary of \Cref{thm:lca_of_image}.

In \Cref{section_algebraic_geometric} we recall how the proof of \Cref{thm:lca_of_image} reduces to the construction of a relative boundary hull of the image of a closed subset $\Gamma \subseteq \AA_Y^n$ under the base-projection $\pi:\AA_Y^n \twoheadrightarrow Y$.

The algorithmic heart of this paper is \Cref{section_proof}, where a new relative boundary hull of $\pi(\Gamma)$ is constructed in Algorithm~\ref{alg:LCA}.
The idea of our construction is to replace $\Gamma \subseteq \AA_Y^n$ by a closed subset\footnote{See \Cref{rmrk:dualidea} for a dual point of view.} $\Gamma_0 \subseteq \Gamma$, that has the same projection closure $\overline{\pi(\Gamma_0)}=\overline{\pi(\Gamma)}$ as $\Gamma$, but where the restriction $\pi_{| \Gamma_0}$ has generically zero-dimensional fibers.
The generic zero-dimensionality of the fibers gives rise to the following argument:
Let $\widehat{\Gamma}_0$ denote the closure of $\Gamma_0 \subseteq \AA_Y^n$ in $\PP_Y^n$ and $H \coloneqq \PP_Y^n \setminus \AA_Y^n$ the hyperplane at infinity.
Then the closed image $D$ of the points $\Gamma_0^\infty \coloneqq \widehat{\Gamma}_0\cap H$ at infinity under the closed projection $\widehat{\pi}: \PP_Y^n \twoheadrightarrow Y$ is a relative boundary hull of $\pi(\Gamma)$.

As mentioned above, constructing a relative boundary hull $D$ is a central step of all previous proofs of Chevalley's theorem known to us.
To simplify the notation one can assume without loss of generality that $\pi_{|\Gamma}$ is dominant, i.e., that $\overline{\pi(\Gamma)}=Y$.
Most of these approaches construct an open set $U=Y\setminus D \subseteq Y$ with a property guaranteeing the inclusion $U \subseteq \pi(\Gamma)$.
However, these properties tend to be far too strong,  guaranteeing much more than the inhabitedness of the fibers over $U$.
This usually results in an unnecessarily large relative boundary hull of high degree and many irreducible components, slowing down the recursion significantly.
One such property is \emph{generic freeness}, where $\OO_{\pi_{|\Gamma}^{-1}(U)}$ is required to be a non-zero free $\OO_U$-module.
The idea of using local freeness was for example exploited in
\begin{itemize}
	\item \cite[Alg.~10.3]{kemper2010course} by an approach using Gröbner bases,
	\item \cite{thomasalg_jsc} by an approach based on resultant, or
	\item Chevalley's original proof \cite[Thm.~3]{chevalley1955schemas}.
\end{itemize}
Similarly, an even more elaborate approach are (canonical) Gröbner covers \cite{MW10}, where the preimage $\pi^{-1}(U)$ can be described by a single Gröbner basis with leading ideal independent of specializations to fibers $\pi^{-1}(u)$ for $u\in U$.
Alternatively, the \emph{Noether normalization} lemma suggests choosing $U$ such that $\pi_{|\pi^{-1}(U)}$ is the composition of a finite surjective morphism $\pi^{-1}(U) \twoheadrightarrow \AA_U^m$ with the projection $\AA_U^m \twoheadrightarrow U$, cf., e.g., \cite[Thm.~1.8.4]{EGA4}.
For a more detailed comparison of these and more approaches we refer the reader to \Cref{section_further_approaches}, where we also argue for the efficiency of our algorithm.

The recent approach \cite{ComputingImages} is similar to ours in that it specifically targets small relative boundary hulls in a computationally efficient algorithm.
It is tailor-made for rational morphisms between projective varieties, whereas our approach is tailor-made for the affine case.
For a detailed comparison see \Cref{subsection_Leipzig}.
The only non-intrinsic step in this and in our approach lies in the reduction $\Gamma \leadsto \Gamma_0$, i.e., the reduction to the case of (generically) zero-dimensional fibers.

Note that all algorithmic approaches to Chevalley's Theorem have to start by computing the closure of the image.
This relies on elimination and usually remains an expensive first step.
Whereas our algorithm theoretically needs primary decompositions, we manage to postpone (or even prevent, cf.\ \Cref{section_no_primary}) primary decompositions; any primary decomposition is computed after passing from $\Gamma$ to $\Gamma_0$ (cf.\ \Cref{subsection_algebraic_dimension}).

An important special case of constructible images are orbits of algebraic group actions $\alpha: Y \times G \to Y$.
The orbit $yG$ of $y \in Y$ is the image of the action map $\alpha_y: G \to Y, g \mapsto yg$.
Chevalley's Theorem implies that each such orbit is even locally closed (\Cref{prop:Borel}).
Group orbits play an important role in the representation theory of semisimple algebraic groups (cf.~\cite{Jantzen04} for nilpotent orbits).
Another prominent example is quantum information theory, where group orbits classify entangled states of multiple qubit systems.
In Section~\ref{sec:orbits} we show how the special context of algebraic group orbits allows for various improvements in our algorithm (cf.~\Cref{coro_intersect_not_iterate} and \Cref{prop_closure_pi_Gamma_0}).
These improvements speed up computations considerably.

As mentioned above, the Gröbner basis algorithm remains the bottleneck when treating bigger examples.
However, we expect that the limits of the Gröbner basis machinery can be pushed much further for $G$-equivariant problems (like the computation of algebraic group orbits) once the $G$-equivariance can be exploited in the algorithm \emph{and} in the underlying data structures.
We will address this in future work.

The structure of the paper is as follows.
\Cref{section_constructible} lays the topological framework for constructible sets, 
which is used in \Cref{section_algebraic_geometric} to rephrase and prove Chevalley's \Cref{thm:Chevalley_general} in terms of projections.
Our relative boundary hull is constructed in \Cref{section_proof} and compared to other relative boundary hulls in \Cref{section_further_approaches}.
In \Cref{section_examples} we reinterpret (generalizations of) Rabinowitsch's trick in our setting (cf.\ \Cref{exmp:Rabinowitsch}).
We demonstrate the algorithmic efficiency of our approach on the uniform matrix product states in \Cref{exmp:uMPS}, significantly improving on the state of the art in \cite{uMPS} by providing an explicit description of $\operatorname{uMPS}(2,2,5)$ and recomputing $\operatorname{uMPS}(2,2,4)$ in seconds without using any additional representation-theoretic arguments.
We 
present a formal Gröbner basis proof for the dimension reduction in \Cref{section_exist_planes} and show how to completely avoid the primary decomposition in \Cref{section_no_primary}.
In \Cref{sec:improve_alg_iteration} we show how to replace the linear structure of the classical proofs summarized in Algorithm~\ref{alg:iteration} by a more sophisticated branching in Algorithm~\ref{alg:iteration_graph}, in which unnecessary branches are removed as early as possible.
Such a branched algorithm can easily be parallelized at several places, a planned improvement of our current implementation.

\section{Constructible sets}\label{section_constructible}

Let $Y$ be a topological space.
A subset $A \subseteq Y$ is said to be
\begin{itemize}
	\item \textbf{locally closed} if one of the following equivalent descriptions holds:
	\begin{itemize}
		\item $A$ is the intersection of an open and a closed subset;
		\item $A$ is the difference of two closed subsets;
		\item $A$ is the difference of two open subsets.
	\end{itemize}
	\item \textbf{constructible} if it is the \emph{finite} union of locally closed sets.
\end{itemize}

We are interested in constructible sets as they describe images of rational maps by Chevalley's \Cref{thm:Chevalley}.

\begin{exmp}\label{example_image}
	Let $k$ be a field or even $k=\Z$.
	The image of the polynomial map
	\begin{align*}
		\varphi:
		\begin{cases}
		\AA_k^2 &\to \AA_k^2, \\
		(x_1,x_2) &\mapsto (b_1,b_2) \coloneqq (x_1, x_1x_2) \mbox{.}
		\end{cases}
	\end{align*}
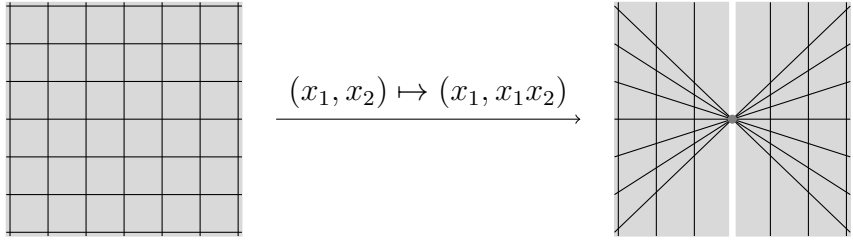
	\begin{center}
		\begin{tikzpicture}[scale=0.5]
		
		\filldraw[gray!30] (-3.1,-3.1) rectangle (3.1,3.1);
		
		\draw[] (-3.1,-3) -- (3.1,-3);
		\draw[] (-3.1,-2) -- (3.1,-2);
		\draw[] (-3.1,-1) -- (3.1,-1);
		\draw[] (-3.1,0) -- (3.1,0);
		\draw[] (-3.1,1) -- (3.1,1);
		\draw[] (-3.1,2) -- (3.1,2);
		\draw[] (-3.1,3) -- (3.1,3);
		
		\draw[] (-3,-3.1) -- (-3,3.1);
		\draw[] (-2,-3.1) -- (-2,3.1);
		\draw[] (-1,-3.1) -- (-1,3.1);
		\draw[] (0,-3.1) -- (0,3.1);
		\draw[] (1,-3.1) -- (1,3.1);
		\draw[] (2,-3.1) -- (2,3.1);
		\draw[] (3,-3.1) -- (3,3.1);
		
		\coordinate (r) at (16,0);
		
		\draw[->] ($0.25*(r)$) -- node[above] {$(x_1,x_2) \mapsto (x_1, x_1x_2)$} ($0.75*(r)$);
		
		\filldraw[gray!30] ($(-3.1,-3.1)+(r)$) rectangle ($(-0.1,3.1)+(r)$);
		\filldraw[gray!30] ($(0.1,-3.1)+(r)$) rectangle ($(3.1,3.1)+(r)$);
		
		\draw[] ($(-3.1,-3)+(r)$) -- ($(3.1,3)+(r)$);
		\draw[] ($(-3.1,-2)+(r)$) -- ($(3.1,2)+(r)$);
		\draw[] ($(-3.1,-1)+(r)$) -- ($(3.1,1)+(r)$);
		\draw[] ($(-3.1,0)+(r)$) -- ($(3.1,0)+(r)$);
		\draw[] ($(-3.1,1)+(r)$) -- ($(3.1,-1)+(r)$);
		\draw[] ($(-3.1,2)+(r)$) -- ($(3.1,-2)+(r)$);
		\draw[] ($(-3.1,3)+(r)$) -- ($(3.1,-3)+(r)$);
		
		\draw[] ($(-3,-3.1)+(r)$) -- ($(-3,3.1)+(r)$);
		\draw[] ($(-2,-3.1)+(r)$) -- ($(-2,3.1)+(r)$);
		\draw[] ($(-1,-3.1)+(r)$) -- ($(-1,3.1)+(r)$);
		\filldraw[gray] ($(0,0)+(r)$) circle (3pt);
		\draw[] ($(1,-3.1)+(r)$) -- ($(1,3.1)+(r)$);
		\draw[] ($(2,-3.1)+(r)$) -- ($(2,3.1)+(r)$);
		\draw[] ($(3,-3.1)+(r)$) -- ($(3,3.1)+(r)$);
		\end{tikzpicture}
	\end{center}
	is the union of $\AA^2 \setminus \{ b_1 = 0\}$ and the origin $\{(0,0)\}$ (cf.\ \Cref{example_plane_line_point}).
	The image is clearly neither closed, nor open, not even locally closed in the Zariski topology.
	However, it is constructible.
\end{exmp}	

For a constructible set $A$, we use the following two approximations to the boundary $\partial A$ to inductively approach the intricate details of $\partial A$, leading to efficient algorithms.

\begin{defn}[Relative boundary]
	The \textbf{relative boundary} $\underline{\partial}$ of $A \subseteq Y$ is defined by
	\begin{align*}
		\underline{\partial} A \coloneqq \overline{\overline{A} \setminus A} = \overline{\overline{A} \cap (Y \setminus A)} \subseteq \overline{A} \cap \overline{Y \setminus A} \eqqcolon \partial A.
	\end{align*}
	If $A$ is a nonempty set we call a closed set $D$ a \textbf{relative boundary hull} of $A\subseteq Y$ if
	\begin{align*}
		\overline{A} \setminus A \subseteq D \subsetneq \overline{A} \mbox{.}
	\end{align*}
	In other words, $D$ contains $\underline{\partial} A$ but does not contain $A$.
	We further define the empty set to be the relative boundary hull of itself.
\end{defn}

The notion of a relative boundary hull is introduced in \cite{ComputingImages} under the name ``frame''.

\begin{defn}[Locally closed part]
	The \textbf{locally closed part} $\operatorname{LCP}(A)$ of an arbitrary subset $A \subseteq Y$ is defined as
	\begin{align*}
		\operatorname{LCP}(A) \coloneqq \overline{A} \setminus \underline{\partial} A.
	\end{align*}
	If $D$ is a relative boundary hull of $A$ then we call $\overline{A} \setminus D$ a \textbf{locally closed approximation} of $A$, the locally closed part being the largest such approximation.
\end{defn}

In \Cref{example_image} the minimal relative boundary hull is the relative boundary $\{ b_1 = 0\}$.
Relative boundary hulls do not always exist, for example for  $A\coloneqq\Q\subset\R\eqqcolon Y$ equipped with the natural topology the equality $\underline{\partial}A=\overline{A}$ holds.
However, relative boundary hull, or equivalently locally closed approximations, exist for images of rational maps in the Zariski topology by \Cref{thm:lca_of_image}, which will be proved in \Cref{section_proof}.

\begin{defn}[Canonical form]\label{definition_canonical_form}
	If $A$ is a constructible set in the topological space $Y$, then the locally closed part $\operatorname{LCP}(A)$ is the largest locally closed subset of $A$.
	If moreover $Y$ is Noetherian, then the \textbf{canonical form} or \textbf{canonical decomposition} of a constructible set is the finite \emph{disjoint} union
	\begin{align*}
		A = \operatorname{LCP}(A) \uplus \operatorname{LCP}( A \setminus \operatorname{LCP}(A) ) \uplus \cdots\mbox{.}
	\end{align*}
\end{defn}

In \Cref{example_image} the locally closed part is $\AA^2 \setminus \{ b_1 = 0\}$, and the canonical form is
\[
  \AA^2 \setminus \{ b_1 = 0\}\uplus\{(0,0)\} \mbox{.}
\]

\begin{rmrk} \label{rmrk:canonical_decomp}
The canonical form is in an obvious sense the ``most exhaustive'' 
decomposition of a constructible set into locally closed subsets.
It is by definition an \emph{intrinsic} decomposition of a constructible set and as such of general interest independent of the context of this paper.
We will therefore defer its algorithmic treatment to a later paper, where we also show its finiteness.
We have already implemented all operations of the Boolean algebra of constructible objects in a locale including the computation of the canonical decomposition in the $\mathsf{GAP}$ package $\mathtt{Locales}$ \cite{Locales}.
The package $\mathtt{Locales}$ uses the philosophy of category constructors \cite{PosCCT} and relies on the \textsc{Cap} project which makes categories accessible to the computer \cite{GPSSyntax,GP_Rundbrief}.
\end{rmrk}

\begin{rmrk}
Let $A=\biguplus_{i=1}^n A_i\setminus B_i$ be the canonical decomposition of the constructible set $A$.
The tuple $(A_1, B_1, \ldots, A_n, B_n)$ of closed subsets is intrinsically associated to $A$.
Hence, any invariant of closed sets induces a $2n$-tuple of invariants when applied to the above tuple.
The dimension, the degree, and the Euler characteristic are such invariants.
An even finer invariant which incorporates all these three invariants is the Chern-Schwartz-MacPherson class  of an affine or projective variety \cite{PleskenCounting,EulerCharacteristicAlgebraicSets,Jos15,BaechlerPlesken,baechlerphd}.
\end{rmrk}

All known image algorithms can a priori only compute relative boundary hulls instead of relative boundaries for the yet unknown image, at least when the morphism has non-zero dimensional fibers.
This means that they can a priori only approximate the canonical decomposition by successive locally closed approximations (LCA, cf.\  Algorithm~\ref{alg:LCA})
\begin{align*}
	A = \operatorname{LCA}(A) \uplus \operatorname{LCA}( A \setminus \operatorname{LCA}(A) ) \uplus \cdots\mbox{.}
\end{align*} 
The result is thus a non-exhaustive decomposition of the constructible image\footnote{However, once the image is written in this way, the canonical decomposition can then be computed a posteriori (see \Cref{rmrk:canonical_decomp}).}.
More precisely:

The decomposition in Formula \eqref{eq:sum} approximates the canonical form by replacing the locally closed part $\operatorname{LCP}(f(C))$ by a locally closed approximation $\operatorname{LCA}(f(C))\coloneqq\overline{f(C)} \setminus D$ of the image $f(C)$, as described above.

The canonical form can be achieved in a single step in the case of algebraic group orbits, see \Cref{prop:Borel} and \Cref{coro_intersect_not_iterate}.

\section{Notational preliminaries}\label{section_algebraic_geometric}

In this section we fix the notation for the remaining paper, both geometrically and algebraically.
The remaining paper mostly uses the language of algebraic geometry.
Additionally, we sometimes use the algebraic language both to specify slight restrictions on the general geometric setting to allow for algorithms and to describe some algorithmic improvements.

To compute the image of a map $f: X \to Y$ of sets it is sufficient to compute the image of the \emph{graph} $\Gamma \coloneqq \Gamma_f \coloneqq \{ (f(x),x) \mid x \in X \} \subseteq Y \times X$ under the projection map $\pi: Y \times X \twoheadrightarrow Y$.
Equally general, one can consider the projection of a subset $\Gamma \subseteq Y \times X$ which is not necessarily the graph of a map.

In our affine setup we consider Zariski-closed subsets $X \subseteq \AA_k^n$ and $Y \subseteq \AA_k^m$ over a commutative coefficients ring $k$.
We write $\Spec B$ for either $\AA_k^m$ or $Y$.
In the first case $B$ is the polynomial algebra $k[b_1, \ldots, b_m]$ and in the second the affine $k$-algebra $k[b_1, \ldots, b_m] / \mathcal{I}(Y)$, where $\mathcal{I}(Y)$ denotes the vanishing ideal of $Y$.
Hence, both $\Gamma$ and $Y \times X$ are closed subsets of the relative affine space $\AA_B^n = \Spec B \times_k \AA_k^n = \Spec R$ for $R:=B[x_1,\ldots,x_n]$.
The above projection $\pi$ is then the restriction of the base-projection $\AA^n_B \twoheadrightarrow \Spec B$ to $Y \times X$.
If $X$ is reduced resp.\  irreducible, then so is $\Gamma_f$.
And if $\Gamma$ is reduced resp.\  irreducible, then so is $\overline{\pi(\Gamma)}$.

As manifested in our use of the $\Spec$-notation our topological spaces are the prime spectra of affine $k$-algebras.
As customary, this is one of the ways to avoid the assumption that $k$ is an algebraically closed field, which is far too restrictive for our intended applications.
However, when our algorithm reduces to the case that $k$ is a finite field we will occasionally pass to a finite extension of $k$ in order to guarantee the existence of certain hyperplanes.
Still, the image $\pi(\Gamma)$ will be a constructible subset of the topological space $\Spec B$, where $B$ is the original $k$-algebra.

Note that non-reduced structures play no role in our setup and we may assume all ideals to be radical, i.e., all affine varieties to be reduced.

Our goal is an algorithm, hence we assume that the rings $k$, $B$, and $R$ allow for Gröbner basis algorithms.
One general framework to allow for Gröbner bases is sketched in \Cref{appendix_computable_ring}, which allows for a wide range of rings $k$ like affine algebras over a constructive field or $\Z$.
We also refer to \Cref{appendix_computable_ring} for a discussion about rings over which our algorithm is possible, even though these rings do not allow Gröbner bases.
We treat Gröbner bases as black boxes, in particular we are agnostic to term orders, except for the occasional need of elimination orders, and do not comment on possible optimization of Gröbner bases for input of certain types.
For computational efficiency we also assume the existence of a primary decomposition algorithm, even though this is not strictly necessary (cf.\ \Cref{section_no_primary}).

If 
\begin{align*}
	f=\frac{p}{q}:\begin{cases}
		X &\to Y\\
		(x_1,\ldots,x_n) &\mapsto (b_i)\coloneqq (f_i(x_1,\ldots,x_n)) = \left(\frac{p_i(x_1,\ldots,x_n)}{q_i(x_1,\ldots,x_n)}\right)
	\end{cases}
\end{align*}
is a rational map\footnote{For rational maps we assume $k$ to be a field.} (or a polynomial map for $q_i\equiv1$), then its graph $\Gamma_f$ is given by the ideal
\begin{align*}
	I\coloneqq\langle q_i(x_1,\ldots,x_n)b_i-p_i(x_1,\ldots,x_n)|i=1,\ldots,m\rangle\unlhd R.
\end{align*}

Encoding maps in projections we can reformulate \cref{thm:lca_of_image} and \Cref{thm:Chevalley} as follows:

\begin{theorem} \label{thm:lca_of_projection}
	$\pi(\Gamma)$ admits a relative boundary hull, or equivalently, a locally closed approximation.
\end{theorem}

\begin{theorem}[Chevalley] \label{thm:Chevalley_projection}
	$\pi(\Gamma)$ is a constructible set.
\end{theorem}

The proof of \Cref{thm:Chevalley_projection} is again a reformulation of the proof of \Cref{thm:Chevalley} given in the Introduction.
We present it in form of the following algorithm, which we express geometrically as Algorithm~\ref{alg:iteration} and, with corresponding lines, algebraically as Algorithm~\ref{alg:iteration_algebraic}.
Note that the used sub-algorithm \textbf{LocallyClosedApproximationOfProjection} depends on the approach used to compute the relative boundary hull:
\begin{itemize}
  \item Our new approach is expressed geometrically as Algorithm~\ref{alg:LCA} and algebraically as Algorithm~\ref{alg:LCA_algebraic}, where the respective algorithms have matching line numbers to indicated the correspondence between geometry and algebra.
  \item Algorithm~\ref{alg:Kemper} \cite[Alg.~10.3]{kemper2010course} uses Gröbner bases to construct a relative boundary hull $D$ where $\Gamma$ has free fibers on the complement $U = \overline{\pi(\Gamma)} \setminus D$.
\end{itemize}
\begin{center}
\begin{algorithm}[H]
        \SetKwIF{If}{ElseIf}{Else}{if}{then}{elif}{else}{}%
        \DontPrintSemicolon
        \SetKwProg{ConstructibleProjection}{ConstructibleProjection}{}{}
        \LinesNotNumbered
        \KwIn{
        A closed subset $\Gamma \subseteq \AA_B^n \xtwoheadrightarrow{\pi} \Spec B$
        }
        \KwOut{
        $\pi(\Gamma) = C$ as a finite \emph{disjoint} union of locally closed subsets of $\Spec B$
        }
        \ConstructibleProjection(){($\Gamma$)}{
                \nl $C \coloneqq \emptyset$\;
                \nl \While{$\Gamma \neq \emptyset$}{
                \nl $A, D \coloneqq \textbf{LocallyClosedApproximationOfProjection}(\Gamma)$ \tcp*{$A \coloneqq \overline{\pi(\Gamma)}$}
                \nl $C \coloneqq C \,\uplus (A \setminus D)$\;
                \nl $\Gamma \coloneqq \Gamma \cap \pi^{-1}(D)$\;\label{alg:iteration:5}
                }
                \nl \Return $C$\;
        }
\caption{Outer induction computing a constructible projection (geometric)}
\label{alg:iteration}
\end{algorithm}
\end{center}

\begin{center}
	\begin{algorithm}[H]
		\SetKwIF{If}{ElseIf}{Else}{if}{then}{elif}{else}{}%
		\DontPrintSemicolon
		\SetKwProg{ConstructibleProjection}{ConstructibleProjection}{}{}
		\LinesNotNumbered
		\KwIn{
		$I \unlhd B[x_1,\ldots,x_n]$ with $\Gamma = V(I) \subseteq \AA_B^n \xtwoheadrightarrow{\pi} \Spec B$
		}
		\KwOut{
		A finite list $C \coloneqq \left((I_1, J_1), (I_2, J_2), \ldots \right)$ of pairs of ideals in $B$ such that $\pi(\Gamma) = \bigcup \left(V(I_i) \setminus V(J_i)\right)$
		}
		\ConstructibleProjection(){($I$)}{
			\nl $C \coloneqq ()$\;
			\nl \While{$1 \notin I$}{
				\nl $I_A, I_D \coloneqq \textbf{LocallyClosedApproximationOfProjection}(I)$\;
				\nl $\mathtt{Add}(C,(I_A, I_D))$\;
				\nl $I \coloneqq I + \langle I_D \rangle_{B[x_1,\ldots,x_n]}$\;
			}
			\nl \Return $C$\;
		}
		\caption{Outer induction computing a constructible projection (algebraic)}
		\label{alg:iteration_algebraic}
	\end{algorithm}
\end{center}

\begin{rmrk}
  All ideals appearing in the algebraic versions of the algorithms (Algorithms~\ref{alg:iteration_algebraic}, \ref{alg:LCA_algebraic}, and \ref{alg:Kemper}) are given by a finite set of generators.
  This is by specification for the input ideals and by construction for all ideals computed within the algorithms.
  It is well-known how to use Gröbner bases and syzygies in order to translate operations like ideal membership test, ideal sum, and ideal quotient into operations involving the corresponding finite sets of generators.
\end{rmrk}

\begin{rmrk} \label{rmrk:decompose_rbh}
  In practice it is extremely useful to exploit known decompositions of the relative boundary hull $D$, and occasionally even to compute a complete decomposition in irreducibles components.
  To this end, the simple linear structure of iteratively constructing $C$ in Algorithm~\ref{alg:iteration} will be replaced by the directed graph structure in Algorithm~\ref{alg:iteration_graph} which we will develop in \Cref{sec:improve_alg_iteration}.
\end{rmrk}

The more general version of Chevalley's \Cref{thm:Chevalley_general} can now be easily and constructively derived from \Cref{thm:Chevalley_projection}:

\subsection{Proof of \Cref{thm:Chevalley_general}}\label{proof_Chevalley_general}

If $\Gamma$ is constructible as the finite union of locally closed subsets $\Gamma_i$, then $\pi(\Gamma) = \bigcup_i \pi(\Gamma_i)$.
Using a generalized version of Rabinowitsch's trick, each locally closed $\Gamma_i \subseteq Z = Y \times X$ can replaced by a so-called \emph{Rabinowitsch cover} which is a \emph{closed} subset in $\Gamma_i^\mathrm{rab} \subseteq Z \times \AA^1 \xrightarrow{\pi'} Z$ with $\pi'\left( \Gamma_i^\mathrm{rab} \right) = \Gamma_i$ (cf.\ \Cref{exmp:Rabinowitsch:exception} in \Cref{exmp:Rabinowitsch}).
The image $\pi(\Gamma_i)$ can now be computed as the image of $\Gamma_i^\mathrm{rab}$ under the composed projection $\pi \circ \pi': Y \times X \times \AA^1 \to Y$.
\qed

\subsection{Computability of rings and Gröbner bases}\label{appendix_computable_ring}

In this paper, we will explicitly or implicitly assume that the three rings $k$, $B \coloneqq k[b_1, \ldots, b_m] / \mathcal{I}(Y)$, and $R \coloneqq B[x_1, \ldots, x_n]$ allow for Gröbner bases.
This will be necessary to compute elimination ideals in Algorithm~\ref{alg:LCA_algebraic} or in various places to decide ideal membership, compute ideal quotients and saturations.

Gröbner bases exist over many rings.
One common construction assumes that $k$ is a ring with \textbf{effective coset representatives} \cite[\S 4.3]{AL}, i.e., if for every ideal $J\unlhd k$ in the commutative coefficients ring $k$ we can determine a set $T$ of coset representatives of $k/J$, such that for every $a \in k$ we can compute a unique $t \in T$ with $a+J=t+J$.
Many rings have this property, e.g., constructive fields and $\Z$.
Furthermore, if $k$ has effective coset representatives, then affine rings over $k$ have effective coset representatives as well.
Hence, if $k$ has effective coset representatives, then there exist Gröbner bases in $B \coloneqq k[b_1, \ldots, b_m] / \mathcal{I}(Y)$ and therefore also in $R \coloneqq B[x_1, \ldots, x_n]$.
In \Cref{section_exist_planes} we also require the decomposition of $k$ into a finite product of domains $k=k_1 \times \cdots \times k_r$ to be constructive.

Theoretically, we could even apply our algorithm over rings without Gröbner bases, as it suffices to assume computability for the rings $k$, $B$, and $R$.
We call a (unital) commutative ring $R$ \textbf{computable} \cite{BL} if there exists an algorithm to solve a linear system over $R$, i.e., to find an (affine) generating set of all $X$ with $b=XA$ for given matrices $A$ and $b$ over $R$.
Of course, any ring with Gröbner bases is computable, but also their residue class rings and certain localizations thereof \cite{BL, BL_ACM, PosLinSys}.
Computability obviously allows to decide ideal membership (a particular solution of a linear system) and to compute ideal quotients and saturations (solutions of homogeneous systems), but also to compute elimination ideals \cite{BL_Elimination}.

According to our experience a primary decomposition algorithm is beneficial for computational efficiency, but otherwise not necessary (cf.\ \Cref{section_no_primary}).

\section{Existence of locally closed approximations and relative boundary hulls}\label{section_proof}

In this section, we present our alternative proof of \Cref{thm:lca_of_projection}, i.e., the existence of locally closed approximations of images of projections or, equivalently, of relative boundary hulls.
We first present a sketch of a purely geometric algorithm for our construction and then explain the algorithm in detail in \Cref{subsection_algebraic_dimension,subsection_relative_boundary_hull}.

\textbf{Idea 1: Take the projective closure of each fiber.}
To understand the origin of the superfluous points $\overline{\pi(\Gamma)} \setminus \pi(\Gamma)$, we ``add the points at infinity'', i.e., we replace $\Gamma \subseteq \AA^n_B \subsetneq \PP^n_B$ ($n>0$) by its Zariski closure $\widehat{\Gamma}$ in $\PP_B^n$ and replace $\pi:\AA^n_B \twoheadrightarrow \Spec B$ by the extended projection $\widehat{\pi}: \PP_B^n \twoheadrightarrow \Spec B$.
We call $\widehat{\Gamma} \subseteq \PP_B^n$ the \textbf{projective closure} of $\Gamma$.
On the one hand $\widehat{\pi}$ is a closed morphism by the \emph{main theorem of elimination theory}, hence the image of the closed subset $\widehat{\Gamma}$ is closed.
One the other hand $\widehat{\pi}(\widehat{\Gamma}) = \overline{\widehat{\pi}(\Gamma)} = \overline{\pi(\Gamma)}$ due to the continuity of $\widehat{\pi}$ and the fact that $\widehat{\Gamma}$ is the closure of $\Gamma$ as a subset of $\PP_B^n$.
It immediately follows that
\begin{equation} \label{eq:incl} \tag{*}
  \underbrace{\overline{\pi(\Gamma)} \setminus \pi(\Gamma)}_{\text{extra points}} \subseteq \underbrace{\widehat{\pi}(\overbrace{\widehat{\Gamma} \setminus \Gamma}^{\eqqcolon \Gamma^\infty})}_{\text{closed}} \subseteq \overline{\pi(\Gamma)} \mbox{,}
\end{equation}
where $\Gamma^\infty \coloneqq \widehat{\Gamma} \setminus \Gamma = \widehat{\Gamma} \cap H$ is the closed set of points of $\widehat{\Gamma}$ at infinity and $H \coloneqq \PP^n_B \setminus \AA^n_B$ is the hyperplane at infinity.

\begin{exmp}\label{example_hyperbola_geometric}
	Consider the hyperbola $\Gamma=\Gamma_0\coloneqq\V(bx-1)$ in $\AA_B^1$ over $B=\Q[b]$.
	\begin{center}
		\begin{tikzpicture}[scale=0.3]
		\draw plot [
		blue,
		samples=100,
		domain=-5:-0.2
		] (\x,{1/\x});
		\draw plot [
		samples=100,
		domain=0.2:5
		] (\x,{1/\x});
		\draw plot [
		samples=100,
		domain=-5:-0.1
		] (\x,{0});
		\draw[-stealth'] (0.1,0) -- (5.4,0) node[below]{\footnotesize$b$};
		\draw (0,0) circle (3pt);
		\end{tikzpicture}
	\end{center}
	The image under the projection $\pi:\Spec(\AA_B^1) \twoheadrightarrow \Spec(B)$ is $\Spec(B)\setminus\{0\}$.
	When approaching the missing point $0\in \overline{\pi(\Gamma)} \setminus \pi(\Gamma)$ in the base, the hyperbola goes to ``infinity'' in the fiber.
\end{exmp}

\textbf{Idea 2: Approximate the image in its closure by computing a relative boundary hull\footnote{The construction of this relative boundary hull is closely connected to our results in \cite{BL_Elimination}, which show that saturations and eliminations are equivalent in a certain sense.}.}
Our goal is to determine $\pi(\Gamma)$ as the difference of the right and the left hand side in \Cref{eq:incl}.
In the optimal case where the first inclusion is an equality, the image is the locally closed subset $\pi(\Gamma) = \overline{\pi(\Gamma)} \setminus \widehat{\pi}(\Gamma^\infty)$.
This includes $\widehat{\pi}(\Gamma^\infty) = \emptyset$ as special instance where the image is closed.
In case the first inclusion in \eqref{eq:incl} is strict then an approximation of the superfluous points $\overline{\pi(\Gamma)} \setminus \pi(\Gamma)$ would require the second inclusion to be strict as well.
In other words, in all cases we need $\widehat{\pi}(\Gamma^\infty)$ to be a relative boundary hull of $\pi(\Gamma)$.

It turns out that Idea 2 is not enough to compute a relative boundary hull, since $\widehat{\pi}(\Gamma^\infty)$ might be all of $\overline{\pi(\Gamma)}$:

\begin{exmp}\label{example_hyperbola2_geometric}
	Reconsider the hyperbola from \Cref{example_hyperbola_geometric} in one fiber-dimension higher, i.e., $\Gamma = V( \langle b x_1 - 1 \rangle ) \subseteq \AA^2_{k[b]}$.
	In this case every nonempty fiber of $\Gamma$ along the projection $\pi: (b, x_1, x_2) \mapsto b$ is $1$-dimensional and the projective closure $\widehat{\Gamma}$ contains for each nonempty fiber a new point at infinity (cf.\ \Cref{example_hyperbola2} for a detailed computation).
	It follows that the second inclusion in \eqref{eq:incl} is an equality and $\widehat{\pi}(\Gamma^\infty)$ is not a relative boundary hull for $\pi(\Gamma)$.
\end{exmp}

\textbf{Idea 3: Make the fibers zero-dimensional.}
To find a suitable relative boundary hull in \Cref{example_hyperbola2_geometric} we need to recover the setup of the original hyperbola, in which the dimensions of the nonempty fibers were (generically) $0$-dimensional.
This can be achieved by replacing $\Gamma$ with the intersection $\Gamma_0 \coloneqq \Gamma \cap L$, where $L \subseteq \AA_B^{n}$ is an affine subspace of appropriate dimension such that the restriction of $\pi$ to $\Gamma_0$ has generically $0$-dimensional fibers but still $\overline{\pi(\Gamma_0)} = \overline{\pi(\Gamma)}$, set-theoretically\footnote{The equality is not a scheme-theoretic equality.}.
Again denoting by $\Gamma_0^\infty \coloneqq \widehat{\Gamma}_0 \cap H$, the above guarantees the strict inclusion
\[
  \overline{\pi(\Gamma)} \setminus \pi(\Gamma) \subseteq \widehat{\pi}(\Gamma_0^\infty) \subsetneq \overline{\pi(\Gamma)}
\]
and $\widehat{\pi}(\Gamma_0^\infty)$ is now a relative boundary hull of the constructible set $\pi(\Gamma)$.
In other words
\[
  \overline{\pi(\Gamma_0)} \setminus \widehat{\pi}(\Gamma_0^\infty)
\]
is a locally closed approximation of $\pi(\Gamma)$.
It can be shown that one can choose $L$ to be constant over the base $\Spec B$.

\textbf{Idea 4: Make $\Gamma$ irreducible.}
The affine subspace $L$ as described above does not necessarily exist if $\Gamma$ has two components of different dimensions.
In that case, we decompose $\Gamma=\Gamma_1\cup\ldots\cup\Gamma_c$ into its irreducible components and intersect each component $\Gamma_i$ independently with an $L_i$.
For computational efficiency, we try to prevent computing a primary decomposition, if possible, by using heuristics (cf.\ \Cref{subsection_algebraic_dimension}).
A complete avoidance of primary decompositions is possible (cf.\ \Cref{section_no_primary}), but not algorithmically prudent.

\subsection{Achieving zero-dimensional fibers}\label{subsection_algebraic_dimension}

In this subsection, we explicitly describe an algorithm to replace the closed set $\Gamma\subseteq \AA_B^n$ by a suitable closed subset having locally zero-dimensional fibers under the projection $\pi:\AA_B^n \twoheadrightarrow \Spec B$, cf.\ Algorithm~\ref{algorithm_zero_dimensional_fibers_heuristic}.
More specifically, we compute
\begin{enumerate}[(i)]
	\item a closed subset $\Gamma_0 \subseteq \Gamma$, such that \label{condition_0_0}
	\item $\Gamma_0$ has locally $0$-dimensional $\pi$-fibers, i.e., over an open set of a component of $\overline{\pi(\Gamma)}$, but still \label{condition_0_1}
	\item $\overline{\pi(\Gamma_0)} = \overline{\pi(\Gamma)}$, i.e., without altering the projection closure. \label{condition_0_2}
\end{enumerate}

The construction of $\Gamma_0$ works as follows.
Assume $\Gamma\subseteq \AA_B^n$ irreducible.
Denote by $d=\dim(\Gamma)-\dim(\overline{\pi(\Gamma)})$, the generic dimension of the fibers of $\pi_{|\Gamma}$.
For a generic affine subspace $L\subseteq\AA_B^n$ of dimension $n-d$ define $\Gamma_0\coloneqq\Gamma\cap L$.
For a generic $p\in\pi(\Gamma)$ it holds that $\dim(\pi_{|\Gamma_0}^{-1}(\{p\})))=0$, since the codimensions of $L$ and $\pi_{|\Gamma}^{-1}(\{p\})$ sum up to $n$.
A formal proof that such an $L$ can be chosen generically is given in \Cref{section_exist_planes}.
The proof is based on Gröbner bases.

In practice, we not do intersect $\Gamma$ with a high codimension $n-d$ subspace $L$, but rather we intersect $\Gamma$ iteratively with $n-d$ hyperplanes.
The choice of these hyperplanes is a Las Vegas algorithm, as we can check whether the choice of the hyperplane is suitable w.r.t.\ assumption \ref{condition_0_2}, i.e., not reducing the image.
For this Las Vegas algorithm, any choice of hyperplane on a Zariski open and dense set is suitable (cf.\ \Cref{section_exist_planes}).

\begin{rmrk}\label{rmrk_efficiency}
	Furthermore, the above construction assumed $\Gamma$ irreducible, i.e., if $\Gamma$ is reducible we should compute its irreducible components and treat those independently.
	However, Algorithm~\ref{algorithm_zero_dimensional_fibers_heuristic} uses two improvements for computational efficiency:
	\begin{itemize}
	  \item it applies heuristics before (and often instead of) the primary decomposition, and
	  \item we do not compute the primary decomposition\footnote{This speed-up conforms to theoretical considerations: degree bounds for Gröbner bases are double exponential in its codimension \cite{mayrritscher}, hence reducing this codimension should drastically reduce the computational complexity.} of $\Gamma$, but of its intersection with hyperplanes (cf.\ $\Gamma_0'$ in line~\ref{algorithm_zero_dimensional_fibers_heuristic_decomposition} of Algorithm~\ref{algorithm_zero_dimensional_fibers_heuristic}).
	\end{itemize}
\end{rmrk}

Only a generic hyperplane $E$ is suitable for an intersection with $\Gamma$ such that the three assumptions above can be achieved.
If the intersection $\Gamma\cap E$ is not of smaller dimension than $\Gamma$, then we need to try a new hyperplane $E$.
If the intersection leads to a smaller image closure, i.e., $\overline{\pi(\Gamma\cap E)}\subsetneq\overline{\pi(\Gamma)}$, the hyperplane $E$ might lead to a decomposition of $\Gamma$.
The first kind of such a decomposition is that the decreasing image leads to a split in the base, i.e., $\overline{\pi(\Gamma\cap E)}$ is the union of components of $\overline{\pi(\Gamma)}$, cf.\ \Cref{exmp_split_in_base_space}.
A second kind of such a decomposition\footnote{For reasons of computationally efficiency, we do not use this second decomposition, but it underlies the complete avoidance of primary decompositions in \Cref{section_no_primary}.} is that the preimage of the image closure leads to a split of $\Gamma$, i.e., $\pi^{-1}(\overline{\pi(\Gamma\cap E)})\cap\Gamma$ is the union of components of $\Gamma$, cf.\ \Cref{exmp_split_in_total_space}.
Finally, if $E$ does not lead to such a decomposition, we discard it as unsuitable, cf.\ \Cref{exmp_no_split}.

\begin{exmp}\label{exmp_split_in_base_space}
	Consider $\Gamma=\{b_1x=0,b_1b_2=0\}\subseteq \AA_{k[b_1,b_2]}^1$ of dimension 2.
	Taking a hyperplane $E=\{x=1\}$ leads to $\Gamma\cap E=\{b_1=0,x=1\}$ of dimension 1.
	However, $\gamma\coloneqq\pi(\Gamma)=\{b_1b_2=0\}$, whereas $\gamma_1\coloneqq\pi(\Gamma\cap E)=\{b_1=0\}$, hence we decrease the image contrary to assumption \ref{condition_0_2}.
	Luckily, $\gamma_2\coloneqq\overline{\gamma\setminus\gamma_1}$ leads to a decomposition of the image, which induces the decomposition $\Gamma=\Gamma_1\cup\Gamma_2$ for $\Gamma_i=\pi^{-1}\left(\gamma_i \right)\cap\Gamma$.
\end{exmp}

\begin{exmp}\label{exmp_split_in_total_space}
	Consider $\Gamma=\{bx=0\}\subseteq \AA_{k[b]}^1$ of dimension 1.
	Taking a random hyperplane $E=\{x=a\}$ for $0\not=a\in k$ leads to $\Gamma\cap E=\{x=a,b=0\}$ of dimension 0.
	However, $\pi(\Gamma)=\Spec k[b]$, whereas $\pi(\Gamma\cap E)=\{b=0\}\subseteq\Spec k[b]$, hence we decrease the image contrary to assumption \ref{condition_0_2}.
	Luckily, $\Gamma_1\coloneqq\pi^{-1}\left(\pi(\Gamma\cap E)\right)\cap\Gamma=\{b=0\}\subseteq\AA_{k[b]}^1$ is a component\footnote{This example shows that we would need a full decomposition into irreducible factors, an equidimensional decomposition does not suffice.} of $\Gamma=\Gamma_1\cup\Gamma_2$ for $\Gamma_2=\overline{\Gamma\setminus\Gamma_1}=\{x=0\}$.
\end{exmp}

\begin{exmp}\label{exmp_no_split}
	Consider $\Gamma=\{x_1-b=0\}\subseteq \AA_{k[b]}^2$ of dimension 2.
	Taking a random hyperplane $E=\{x_1=a\}$ for $a\in k$ leads to $\Gamma\cap E=\{x_1=b=a\}$ of dimension 1.
	However, $\pi(\Gamma)=\Spec k[b]$, whereas $\pi(\Gamma\cap E)=\{b=a\}\subseteq\Spec k[b]$, hence we decrease the image contrary to assumption \ref{condition_0_2}.
	Here, we just took an unlucky hyperplane, but the hyperplane $E=\{x_2=0\}$ works.
\end{exmp}

We formalize the above approach in Algorithm~\ref{algorithm_zero_dimensional_fibers_heuristic} under the assumption that the coefficients ring $k$ is an infinite domain and that the composition $\Gamma \to \Spec B \to \Spec k$ is dominant.
The general case will be reduced to this one in \Cref{section_exist_planes}.

\begin{algorithm}[H]
	\SetKwIF{If}{ElseIf}{Else}{if}{then}{elif}{else}{}%
	\DontPrintSemicolon
	\SetKwProg{ZeroDimensionalFibers}{ZeroDimensionalFibers}{}{}
	\LinesNotNumbered
	\KwIn{
		A closed subset $\Gamma \subseteq \AA_B^n \xtwoheadrightarrow{\pi} \Spec B$.
		We assume $k$ to be an infinite domain and that the composition $\Gamma \to \Spec B \to \Spec k$ is dominant.
	}
	\KwOut{
		A closed affine subset $\Gamma_0'\subseteq\Gamma$ and a (possibly empty) list of additional closed affine subsets $[\Gamma_1',\ldots,\Gamma_e']$ defining a closed affine subset $\Gamma_0 \coloneqq \Gamma_0'\cup \Gamma_1'\cup\ldots\cup \Gamma_e' \subseteq \Gamma$, such that \ref{condition_0_1} and \ref{condition_0_2} are satisfied; more specifically, the fibers in \ref{condition_0_1} are locally zero-dimensional over the subset $\Gamma_0'$ of $\Gamma_0$.
	}
	\ZeroDimensionalFibers(){($\Gamma$)}{
		\nl  $\Gamma_0' \coloneqq \Gamma$\tcp*{candidate $\Gamma_0$, dimension to be decreased below}
		\nl  $s \coloneqq 1$\tcp*{step counter}
		\nl  $\ell \coloneqq [\,]$\tcp*{list of additional components}
		\tcc{decrease the dimension in the fiber, till it is zero:}
		\nl  \While{$\dim(\Gamma_0')-\dim(\pi(\Gamma_0'))>0$}{
			\nl  $E \coloneqq \operatorname{RandomHyperplane}(\AA^n_k)\times_{\Spec k}\Spec B$\tcp*{cf.\ \Cref{rmrk_random_hyperplane}} \label{algorithm_zero_dimensional_fibers_heuristic_random_hyperplane}
			\nl  $\Gamma_0''\coloneqq \Gamma_0'\cap E$\tcp*{intersect with hyperplane}
			\nl  \If(\tcp*[f]{intersection decreases dimension\ldots}){$\dim(\Gamma_0'')<\dim(\Gamma_0')$}{
				\nl  \If(\tcp*[f]{\ldots without reducing the image}) {$\overline{\pi(\Gamma_0')}\subseteq \overline{\pi(\Gamma_0'')}$}{
					\nl $\Gamma_0' \coloneqq \Gamma_0''$\tcp*{found smaller $\Gamma_0'$}
				}
			\nl  \ElseIf(\tcp*[f]{split w.r.t.\ reduced image, avoid early}){$s>n$}{ \label{algorithm_zero_dimensional_fibers_heuristic:10}
				\tcc{try splitting $\Gamma_0'$ by splitting the base space:}
				\nl  $\Delta\coloneqq \overline{\overline{\pi(\Gamma_0')}\setminus\overline{\pi(\Gamma_0'')}}$\tcp*{``complement'' of the reduced image}
				\nl  \If(\tcp*[f]{if the image reduces nontrivially}){$\overline{\pi(\Gamma_0')}\not\subseteq\Delta$}{
					\nl  $\Gamma_0' \coloneqq \pi^{-1}\left(\overline{\pi\left(\Gamma_0''\right)}\right)\cap\Gamma$\tcp*{continue with one part}
					\nl  $\ell\coloneqq \operatorname{Add}(\ell,\pi^{-1}(\Delta)\cap\Gamma)$\tcp*{store the other half}
					}
					\ElseIf(\tcp*[f]{very expensive, avoid early}){$s>4n$}{
						\tcc{find components:}
						\nl $P \coloneqq \operatorname{Decomposition}(\Gamma_0')$\tcp*{in irreducibles}\label{algorithm_zero_dimensional_fibers_heuristic_decomposition}
						\nl $\Gamma_0' \coloneqq P_1$\tcp*{continue with first component}
						\nl $\ell \coloneqq \ell\cup P\setminus\{\Gamma_0'\}$\tcp*{store remaining components}\label{algorithm_zero_dimensional_fibers_heuristic:17}
					}
				}
			}
			\nl  $s\coloneqq s + 1$\tcp*{increase number of steps}
		}
		\nl  \Return{$[\Gamma_0',\ell]$}
	}
	\caption{ZeroDimensionalFibers \label{algorithm_zero_dimensional_fibers_heuristic}}
\end{algorithm}

Finding an affine subset with small coefficients is of vital importance for computational efficiency: ``complicated'' hyperplanes not only impede the current step, but might also lead to ``complicated'' relative boundary hulls which massively hamper subsequent steps.

\begin{rmrk}\label{rmrk_random_hyperplane}
	An efficient implementation can choose the hyperplanes in line~\ref{algorithm_zero_dimensional_fibers_heuristic_random_hyperplane} by the following heuristic approach.
	First of all, take the affine hyperplane $E$ to be constant over the base $\Spec B$.
	The first $n$ hyperplanes we test are $\{x_i=0\}$, $1\le i\le n$.
	Then, try several hyperplanes of the form $\{x_i-a\}$ for ``small'' elements $a\in k$.
	Only if this fails, one could entertain hyperplanes of the form $\{x_i-a_1x_j-a_0\}$ for ``small'' elements $a_0,a_1\in k$.
	Note that for finite fields $k$, it might be necessary to take coefficients $a_i$ from finite field extensions.
	We refer the reader to \Cref{section_exist_planes} for a proof (using Gröbner bases) of the generic existence of such hyperplanes.
\end{rmrk}

We employ the (potentially partial) decomposition $\Gamma_0 \coloneqq \Gamma_0'\cup \Gamma_1'\cup\ldots\cup \Gamma_e' \subseteq \Gamma$ created by Algorithm~\ref{algorithm_zero_dimensional_fibers_heuristic} to continue with each component $\Gamma_i'$ independently, cf.\ \Cref{remark_additional_components} for details.

\subsection{Computing the relative boundary hull of the projection}\label{subsection_relative_boundary_hull}

As described in \Cref{subsection_algebraic_dimension}, we begin by replacing $\Gamma$ by 
\begin{enumerate}[(i)]
	\item a closed subset $\Gamma_0 \subseteq \Gamma$, such that
	\item $\Gamma_0$ has locally $0$-dimensional $\pi$-fibers, but still
	\item $\overline{\pi(\Gamma_0)} = \overline{\pi(\Gamma)}$.
\end{enumerate}

The next geometric step is to take the projective closure $\widehat{\Gamma}_0 \subseteq \PP_B^n$ of $\Gamma_0 \subseteq \AA_B^n$.
Now define $\Gamma_0^\infty \coloneqq \widehat{\Gamma}_0\cap H$, the points of $\Gamma_0$ at infinity, where $H \coloneqq \PP_B^n \setminus \AA_B^n$ is the hyperplane at infinity.
The image $\widehat{\pi}(\Gamma_0^\infty)$ of $\Gamma_0^\infty$ under the extended projection $\widehat{\pi}: \PP_B^n \twoheadrightarrow \Spec B$ is the desired relative boundary hull.

\begin{algorithm}[H]
        \SetKwIF{If}{ElseIf}{Else}{if}{then}{elif}{else}{}%
        \DontPrintSemicolon
        \SetKwProg{LocallyClosedApproximationOfProjection}{LocallyClosedApproximationOfProjection}{}{}
        \LinesNotNumbered
        \KwIn{
        A closed subset $\Gamma \subseteq \AA_B^n \xtwoheadrightarrow{\pi} \Spec B$}
        \KwOut{
        The closure $\overline{\pi(\Gamma)}$ of the projection together with a relative boundary hull $D$.
        }
        \LocallyClosedApproximationOfProjection(){($\Gamma$)}{
\setcounter{AlgoLine}{-1}
                \nl $\Gamma_0 \coloneqq \textbf{ZeroDimensionalFibers}(\Gamma)$\;
                \nl Compute $\overline{\pi(\Gamma_0)}$\;
                \nl Compute $\widehat{\pi}(\Gamma_0^\infty)$\;
                \nl \Return{$\overline{\pi(\Gamma_0)}, \widehat{\pi}(\Gamma_0^\infty)$}\;
        }
\caption{Locally closed approximation of projection (geometric)} \label{alg:LCA}
\end{algorithm}

We implemented the above geometric algorithm in the $\mathsf{GAP}$ package $\mathtt{ZariskiFrames}$ \cite{ZariskiFrames}.
It gets automatically translated upon execution into an algebraic algorithm (Algorithm~\ref{alg:LCA_algebraic}) by using the standard dictionary between geometry and algebra, including the following ``optimization'':

In order to compute $\Gamma_0^\infty \coloneqq \widehat{\Gamma}_0\cap H$ we would first need to compute the homogenization $I_0^h$ of the vanishing ideal $I_0$ of $\Gamma_0$ w.r.t.\ a new indeterminate $x_0$ and then intersect the vanishing locus of $I_0^h$ with the hyperplane $\{x_0=0\}$.
We can do the composition of these two steps in an optimized way: just take the monomials of maximal degree (in $x_1, \ldots, x_n$) from the polynomials in $I_0$.
Formally, we denote
\begin{align*}
	f^{\mathrm{maxdeg}} = \sum_{i\in \Z_{\ge0}^m, |i|=d}q_ix^i
\end{align*}
for $f=\sum_{i\in \Z_{\ge0}^m}q_ix^i\in R$, $q_i\in B$, and $d=\max_{q_i\not=0}|i|$ and compute
\begin{align*}
	I_0^{\mathrm{maxdeg}}
	\coloneqq
	\langle
	f^{\mathrm{maxdeg}} \mid f \in I_0
	\rangle
	=
	\langle
	f^{\mathrm{maxdeg}} \mid f \in G_0
	\rangle
	=
	\mcI(\widehat{\Gamma}_0\cap H)
	\mbox{,}
\end{align*}
where $G_0$ is a Gröbner basis of $I_0$ w.r.t.\ a block elimination ordering with $x_i\gg B$ for all $i$.
It follows that
\[
  I_0^{\mathrm{maxdeg}} = \left(I_0^h\right)_{|x_0 = 0} \mbox{.}
\]

We cannot compute the projection of $\Gamma_0^\infty$ down to $\Spec B$ by substituting all $x_i'$s by zero in $I_0^{\mathrm{maxdeg}}$, since $\big(I_0^{\mathrm{maxdeg}}\big)_{|x_1 = \cdots = x_n = 0} = \left(I_0^h\right)_{|x_0 = x_1 = \cdots = x_n = 0} = I_0 \cap B = I\cap B$, where the second and last equality are \cite[Cor.~2.9]{BL_Elimination}.
We rectify this by saturating\footnote{A single ideal quotient might not be sufficient, as $\V\left(bx_1^2-1\right)$ in $\AA_B^1$ over $B=\Q[b]$ shows.} $I_0^{\mathrm{maxdeg}}$ w.r.t.\ the ideal $\langle x_1, \ldots, x_n \rangle$, which corresponds to the irrelevant locus of the hyperplane at infinity.
The vanishing locus of the ideal
\[
  I_\rbh^\infty \coloneqq \left(I^{\mathrm{maxdeg}} : \langle x_1, \ldots, x_n \rangle^\infty \right)_{|x_1 = \cdots = x_n = 0}
\]
is now a relative boundary hull of the image of $\pi_{|\Gamma_0}$.

\begin{algorithm}[H]
	\SetKwIF{If}{ElseIf}{Else}{if}{then}{elif}{else}{}%
	\DontPrintSemicolon
	\SetKwProg{LocallyClosedApproximationOfProjection}{LocallyClosedApproximationOfProjection}{}{}
	\LinesNotNumbered
	\KwIn{
		An ideal $I$ defining an irreducible closed subset $\Gamma \coloneqq V(I) \subseteq \AA^n_B = \Spec R$
	}
	\KwOut{
		An ideal $I_B \unlhd B$ defining the closure of the projection $\overline{\pi(\Gamma)} \coloneqq V(I_B) \subseteq \Spec B$ and an ideal $I_\rbh^\infty \unlhd B$ defining a relative boundary hull $D \coloneqq V(I_\rbh^\infty)$ of $\pi(\Gamma)$.
	}
	\LocallyClosedApproximationOfProjection(){($I$)}{
\setcounter{AlgoLine}{-1}
		\nl  $I_0 \coloneqq \textbf{ZeroDimensionalFibers}(I)$\; \label{alg:LCA_algebraic:0}
		\tcc{The following Gröbner basis $G_0$ gets computed in the algebraic version of Algorithm~\ref{algorithm_zero_dimensional_fibers_heuristic} (\textbf{ZeroDimensionalFibers}).}
		\nl$G_0 \coloneqq\textbf{GröbnerBasis}_{x_i\gg B}(I_0)$\; \label{alg:LCA_algebraic:1}
		$I_B \coloneqq \langle G_0 \cap B \rangle = I_0 \cap B = I \cap B$\;
		\nl $I_0^\mathrm{maxdeg} \coloneqq \langle G_0 ^{\mathrm{maxdeg}}\rangle \coloneqq \left\langle f^{\mathrm{maxdeg}} \mid f \in G_0 \right\rangle$\;
		$I_\rbh^\infty \coloneqq \left( I_0^\mathrm{maxdeg} : \langle x_1, \ldots, x_n \rangle^\infty \right)_{|x_1 = \cdots = x_n = 0}$\; \label{alg:LCA_algebraic:2}
		\nl \Return{$I_B, I_\rbh^\infty$}\;
	}
	\caption{Locally closed approximation of projection (algebraic) \label{alg:LCA_algebraic}}
\end{algorithm}
\noindent
The line numbering in Algorithm~\ref{alg:LCA_algebraic} corresponds to that in Algorithm~\ref{alg:LCA}.

\begin{exmp}\label{example_hyperbola}
	Consider again the hyperbola $\Gamma=\Gamma_0\coloneqq\V(bx-1)$ in $\AA_B^1$ over $B=\Q[b]$ from \Cref{example_hyperbola_geometric}.
	The monomials of maximal degree are $I_0^{\mathrm{maxdeg}}=\langle bx\rangle$.
	This ideal has two components $\langle x\rangle$ and $\langle b\rangle$.
	The projection of the first component $\langle x\rangle$ corresponds to $\overline{\pi(\Gamma)}$ and is removed by the saturation w.r.t.\ $x$.
	The projection $\langle b\rangle$ of the second component $\langle b\rangle$ defines a relative boundary hull, even the relative boundary, of the image.
\end{exmp}

\begin{exmp}\label{example_hyperbola2}
	Consider again the hypersurface $\Gamma\coloneqq\V(bx_1-1)$ in the \emph{two}-dimensional affine space $\AA_B^2$ with coordinate ring $B[x_1,x_2]$ over $B=\Q[b]$ from \Cref{example_hyperbola2_geometric}.
	The variety of the ideal $I^{\mathrm{maxdeg}}=\langle bx_1\rangle$ has all of $\Spec B$ as image under the projection, even after saturation w.r.t.\ $\langle x_1,x_2 \rangle$.
	Intersecting $\Gamma$ with a generic affine subspace $x_2=ax_1$ with $a\in k$ yields $\Gamma_0=\V(bx_1-1,ax_1-x_2)$, which is up to isomorphism the situation of \Cref{example_hyperbola}.
\end{exmp}

\begin{exmp}
	Consider the hypersurface $\Gamma\coloneqq\V(b\cdot (x^2+1)-x)$ in the one-dimensional affine space $\AA_B^1$ with coordinate ring $B[x]$ over $B=\Q[b]$.
	The relative boundary hull of $\pi(\Gamma)$ is given by the variety of the ideal $\langle bx^2\rangle:\langle x \rangle^\infty=\langle b\rangle$.
	Computing further, $\pi\left(\pi^{-1}\left(\V(\langle b\rangle)\right)\cap\Gamma\right)=\V(\langle b\rangle)$.
	Hence, $\pi(\Gamma)=\left(\Spec B\setminus \V(\langle b\rangle)\right)\uplus\V(\langle b\rangle)=\Spec B$.
	So, even though the first relative boundary hull is non-trivial, $\pi_{|\Gamma}$ is surjective.
	
	
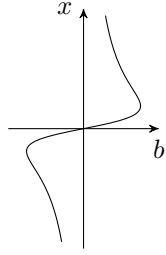
\begin{center}
		\begin{tikzpicture}[scale=0.3]
		\draw plot [
		samples=100,
		domain=-5:5
		] ({5*\x/(1+\x*\x)},\x);
		\draw[-stealth'] (-3.3,0) -- (3.3,0) node[below]{\footnotesize$b$};
		\draw[-stealth'] (0,-5.3) -- (0,5.3) node[left]{\footnotesize$x$};
		\end{tikzpicture}
	\end{center}
\end{exmp}

\begin{rmrk}\label{remark_additional_components}
	The ideal $I_0$ returned by \textbf{ZeroDimensionalFibers} can be returned as several components.
	The additional components are returned in the list $\ell$ in Algorithm~\ref{algorithm_zero_dimensional_fibers_heuristic}.
	In theory, all these components should be intersected to continue.
	In practice, we just continue with the main component $\Gamma_0'$, for which the condition on zero-dimensional fibers was ensured, and treat the remaining components $\Gamma_i'$ in separate cases.
	This has a major advantage in computational efficiency, as we can work with smaller components, at the cost of losing disjointness.
	Disjointness can be restored by additional computations as in \Cref{rmrk:canonical_decomp}.
\end{rmrk}

\begin{rmrk} \label{rmrk:dualidea}
  Let $f: X \to Y$ be a rational map between affine varieties with positive dimensional fibers.
  An, in some sense dual, idea to decrease the (generic) fiber dimension of $f$ is to find a factorization of $f$
  \begin{center}
    \begin{tikzpicture}
      \coordinate (r) at (4,0);
      \coordinate (d) at (0,-1);
      
      \node (S) {$X$};
      \node (T) at ($(r)$) {$Y$};
      \node (X) at ($0.5*(r)+(d)$) {$Z$};
      
      \draw[-stealth'] (S) -- node[above]{$f$} (T);
      \draw[-doublestealth] (S) -- node[below left]{$\varepsilon$} (X);
      \draw[-stealth'] (X) -- node [below right]{$\widetilde{f}$} (T);
    \end{tikzpicture}
  \end{center}
  where $\varepsilon$ is surjective.
  It follows that $f$ and $\widetilde{f}$ have the same image and: the higher the generic fiber dimension of $\varepsilon$, the lower that of $\widetilde{f}$.
  And once $\widetilde{f}$ has even $0$-dimensional fibers (at least generically) Line~\ref{alg:LCA_algebraic:0} of Algorithm~\ref{alg:LCA_algebraic} becomes redundant (cf.~\Cref{exmp:uMPS}).
  Achieving an epi-mono factorization would be optimal (cf.~\Cref{sec:epimono}).
\end{rmrk}

\section{Further approaches to relative boundary hulls}\label{section_further_approaches}

All constructive proofs of Chevalley's Theorem known to us use locally closed approximations via relative boundary hulls.
We compare our approach to some proofs, many of which have been made algorithmic, and argue for greater efficiency of our algorithm.

\subsection{Comparing the computational efficiency of approaches to Chevalley's theorem}

All approaches for the image of a polynomial map make use of Algorithm~\ref{alg:iteration}.
In particular, these algorithms only differ in their respective subalgorithms which compute the relative boundary hull.
To compare the efficiency of these algorithms, we consider the number of iterations in Algorithm~\ref{alg:iteration} and the computation time in each iteration.

The number of iterations in Algorithm~\ref{alg:iteration} is influenced by the codimension of the relative boundary hull $D$ of $\pi(\Gamma)$.
Generically, the input $\Gamma:=\Gamma \cap \pi^{-1}(D)$ (line~\ref{alg:iteration:5} of subsequent steps in Algorithm~\ref{alg:iteration}) has the same codimension as $D$ has in $\overline{\pi(\Gamma)}$.
Since Algorithm~\ref{alg:iteration} terminates as soon as $\Gamma$ is empty, a larger codimension of $D$ is favorable.
If $\overline{\pi(\Gamma)}$ is irreducible, $D$ is at least of, but rarely bigger than, codimension one.
A relative boundary hull $D$ with multiple components usually results in a reducible subsequent $\overline{\pi(\Gamma)}$.
In particular, a relative boundary hull with many components often leads to additional iterations in Algorithm~\ref{alg:iteration}.

In each iteration, the computation time of the relative boundary hull is strongly dominated by computing the closure of the image (e.g.\ by computing the Gröbner basis of line \ref{alg:LCA_algebraic:1} in Algorithm~\ref{alg:LCA_algebraic}).
This is true for our and all other algorithmically relevant approaches to the image of a polynomial map.
The computation time of this closure in subsequent iterations strongly depends on the degree of relative boundary hulls, as e.g.\ higher degrees make Gröbner bases significantly slower.
To further decrease computation time in later iterations, our algorithm retains sparsity in its input, e.g.\ by choosing hyperplanes in a sparse way (\Cref{rmrk_random_hyperplane}), postponing computations after the intersection with (sparse) hyperplanes (\cref{rmrk_efficiency}), and by decompositions.

Decompositions of the sets $\Gamma$ (or $\Gamma_0$) increase the number of iterations, which slightly increases the computation time, but they strongly decrease the size of the considered $\Gamma$ in the iteration.
The slight increase in the number of iterations can --to a certain extent-- be controlled by \Cref{sec:improve_alg_iteration}.
We always use a decomposition if we know one.
As primary decompositions are expensive to compute, we employ heuristics (cf.\ \cref{rmrk_efficiency}, \Cref{section_no_primary}, and lines \ref{algorithm_zero_dimensional_fibers_heuristic:10} to \ref{algorithm_zero_dimensional_fibers_heuristic:17} in Algorithm~\ref{algorithm_zero_dimensional_fibers_heuristic}).

Summing up, the efficiency of a proof of Chevalley's theorem can mostly be read off the degree of its relative boundary hulls.

Many approaches construct a relative boundary hull $D\subseteq Y$ of the image of a (w.l.o.g.) dominant morphism $f:X\to Y$ by imposing additional strong properties on the open complement $U=Y\setminus D= \overline{f(X)} \setminus D$, which as a by-product ensures $U\subseteq f(X)$.
This deviates the focus from constructing a small relative boundary hull $D$.
Instead of ensuring such properties for $U$, our algorithm focuses on constructing points potentially outside of the image, by finding points coming from infinity.
Furthermore, our approach is intrinsic in the case of zero-dimensional fibers, whereas previous approaches depend on a choice $\alpha$ (e.g., coordinate system, term order, \ldots), so that the relative boundary hull $D=D(\alpha)$ could be replaced by the smaller $\bigcap_\alpha D(\alpha)$.
And even if one succeeds in intersecting over different choices, the additional properties imposed on $U$ force the degree of the relative boundary hull $D$ to be larger than the one produced by our approach (for an explicit comparison see \Cref{rmrk:comparison} and the examples following it).
This in turn induces more costly subsequent computations and more iterations in Algorithm~\ref{alg:iteration}.

The following remark describes an extreme case of ensuring stronger properties.

\begin{rmrk}\label{remark_second_induction}
	Whereas the outer induction of Chevalley's proof in Algorithm~\ref{alg:iteration} seems necessary to compute the image of a rational map, some constructions of a relative boundary hull rely on an \emph{additional} induction on the dimension in order to ensure the generic zero-dimensionality of the fibers of $\Gamma$ and its projections in each step.
	To this end the base-projection $\pi: \AA_B^n \twoheadrightarrow \Spec B$ is factored into $n$ successive projections $\AA_B^n \twoheadrightarrow \AA_B^{n-1} \twoheadrightarrow \cdots \twoheadrightarrow \AA_B^1 \twoheadrightarrow \AA_B^0 = \Spec B$.
	Such an induction on the dimension is for example necessary for all resultant based approaches that need to interpret $\AA^i_B$ as $\AA_{\AA_B^{i-1}}^1$.
	This adds major computational drawbacks.
	Already the first projection $\AA_B^n \twoheadrightarrow \AA_B^{n-1}$ of a closed $\Gamma \subseteq \AA_B^n$ is in general a constructible subset of $\AA_B^{n-1}$, introducing relative boundary hull so ``big'' that even using decompositions often fails due to a combinatorial explosion of the number of cases.
\end{rmrk}

\subsection{Generic freeness approaches}

Generic freeness is at the heart of two constructions of relative boundary hulls: via Gröbner bases and via resultants.
A special case of the generic freeness lemma can be stated as follows:

\begin{theorem}[Generic freeness lemma for affine rings]
	Let $R=B[x_1,\ldots,x_n]$ for a commutative Noetherian domain $B$.
	Then, for any finitely presented $R$-module $M$ there exists an $s\in B$ such that the localization $M_s$ is a free $B_s$-module.
\end{theorem}

Consider $\Gamma\subseteq \AA^n_B$.
As mentioned above, we may assume without loss of generality that the projection $\pi:\AA_B^n \twoheadrightarrow \Spec B$ is dominant when restricted to $\Gamma$.
Then, the module $M=R/\mcI(\Gamma)$ is free over $\Spec(B)\setminus\V(s)$, for $s\in B$ as in the generic freeness lemma.
As $M$ is additionally non-zero due to the dominance, $\Spec(B)\setminus \V(s)$ is part of the image of $\Gamma$ under $\pi$.
Hence, $\V(s)$ is a relative boundary hull.

Local freeness of positive rank is a rather strong property, which is only needed to guarantee non-zero fibers.
This leads to computational inefficiencies due to the polynomial $s$ often being of a very high degree, as not only the points not in the image are described by $s$, but also the (usually much bigger) locus of non-freeness.

\subsubsection{A Gröbner basis approach to generic freeness}

Generic freeness is the idea of \cite[Alg.~10.3]{kemper2010course}, where Gröbner basis structures over $R=B[x_1,\ldots,x_n]$ are used to construct a ``Gröbner relative boundary hull'', see Algorithm~\ref{alg:Kemper} for details.
This approach computes the
reduced 
Gröbner basis $G:=\text{GröbnerBasis}_{x_i\gg B}(I)$ w.r.t.\ an elimination ordering $x_i\gg B$.
Now, a relative boundary hull is given by $\V(s)$ for
\begin{align*}
	s=\prod_{g\in (G\setminus B)}\operatorname{LC}_B(g)\mbox{,}
\end{align*}
where $\operatorname{LC}_B(g)$ denotes the leading coefficient as a polynomial in $B$ of the polynomial $g$ when considered as a polynomial in the $x_i$'s w.r.t.\ the chosen monomial order $x_i\gg B$ in the Gröbner basis calculation.
This approach yields a new algorithmic proof of generic freeness for rings with Gröbner bases.
Note that this approach even needs to make two choices: the choice of coordinates and the choice of term orders.
A special case of this approach for $B$ a principal ideal domain (specifically $B=\Z$) is given in \cite[Thm.~3.5]{BV03}.

\begin{algorithm}[H]
	\SetKwIF{If}{ElseIf}{Else}{if}{then}{elif}{else}{}%
	\DontPrintSemicolon
	\SetKwProg{LocallyClosedApproximationOfProjection}{LocallyClosedApproximationOfProjection}{}{}
	\LinesNotNumbered
	\KwIn{
		An ideal $I$ defining an irreducible closed subset $\Gamma \coloneqq V(I) \subseteq \AA^n_B = \Spec R$ where $R \coloneqq B[x_1, \ldots, x_n]$
	}
	\KwOut{
		An ideal $I_B \unlhd B$ defining the closure of the projection $\overline{\pi(\Gamma)} \coloneqq V(I_B) \subseteq \Spec B$ and an ideal $I_\rbh \unlhd B$ defining a relative boundary hull $D \coloneqq V(I_\rbh)$ of $\pi(\Gamma)$.
	}
	\LocallyClosedApproximationOfProjection(){($I$)}{
		\nl $G \coloneqq\textbf{GröbnerBasis}_{x_i\gg B}(I)$ \\
		$I_B \coloneqq \langle G \cap B \rangle = I \cap B$\;
		\nl $I_\rbh \coloneqq \left\langle \prod_{g\in (G\setminus B)} \operatorname{LC}_B(g) \right\rangle + I_B = \left(\prod_{g\in (G\setminus B)} \left\langle  \operatorname{LC}_B(g) \right\rangle\right) + I_B$\; \label{alg:Kemper:2}
		\nl \Return{$I_B, I_\rbh$}\;
	}
	\caption{Locally closed approximation of projection via freeness (algebraic)\label{alg:Kemper}}
\end{algorithm}
\noindent
The line numbering in Algorithm~\ref{alg:Kemper} corresponds to that in Algorithm~\ref{alg:LCA_algebraic}.

This relative boundary hull has three obvious advantages:
\begin{itemize}
  \item its computation does not require generically zero-dimensional fibers as in Line~\ref{alg:LCA_algebraic:0} of Algorithm~\ref{alg:LCA_algebraic};
  \item it is easy to compute since we need the elimination Gröbner basis anyway for the computation of the ideal $I_B$ with vanishing locus $\overline{\pi(\Gamma)}$;
  \item the factorization $\prod_{g\in (G\setminus B)}\operatorname{LC}_B(g)$ in Line~\ref{alg:Kemper:2} induces an obvious decomposition of the relative boundary hull (see below).
\end{itemize}

However, and apart from the non-intrinsic choices necessary for this construction, the resulting relative boundary hull is of equal dimension as    the one proposed in Algorithm~\ref{alg:LCA_algebraic} of this paper but tends to be of much higher degree (and hence much larger).
This can be seen as follows:

\begin{rmrk} \label{rmrk:comparison}
Algorithm~\ref{alg:Kemper} becomes directly comparable to our new Algorithm~\ref{alg:LCA_algebraic} when the ambient fiber dimension $n=1$, i.e., when $R=B[t]$.
In that case Line~\ref{alg:LCA_algebraic:0} of Algorithm~\ref{alg:LCA_algebraic} becomes obsolete and Line~\ref{alg:LCA_algebraic:2} of Algorithm~\ref{alg:LCA_algebraic} simplifies to
\[
  I_\rbh^\infty \coloneqq I_B + \left\langle \operatorname{LC}_B(g) \mid g \in (G \setminus B) \right\rangle = I_B + \sum_{g \in (G \setminus B)} \left\langle \operatorname{LC}_B(g) \right\rangle \mbox{.}
\]
In particular, as $I_\rbh^\infty=I_B+\sum_{g \in (G \setminus B)} \left\langle \operatorname{LC}_B(g) \right\rangle \supseteq I_B+\prod_{g\in (G\setminus B)} \left\langle  \operatorname{LC}_B(g) \right\rangle=I_\rbh$, we get the inclusion
\[
  D^\infty \coloneqq V(I_\rbh^\infty) \subseteq V(I_\rbh) \eqqcolon D
\]
for the two relative boundary hulls $D^\infty$ and $D$.
Furthermore, the degree of $I_\rbh$ increases rapidly with the number of Gröbner basis elements in $G \setminus B$.
\end{rmrk}

\begin{rmrk} \label{rmrk:improveKemper}
To improve Algorithm~\ref{alg:Kemper} one can use more computations to replace $I_\rbh$ with its radical $\sqrt{I_\rbh}$ at Line~\ref{alg:Kemper:2} in the obvious way:
Denote by $M$ the set of (normalized) irreducible factors of all leading coefficients $\operatorname{LC}_B(g)$ for $g \in (G \setminus B)$. Then $\prod_{c \in M} c$ is the square-free part of $\prod_{g\in (G\setminus B)} \operatorname{LC}_B(g)$.
Further denote by $N$ be the set of minimal associated primes of all ideals $\langle c, I_B \rangle$.
Then $\sqrt{I_\rbh} = \bigcap_{\p \in N} \p$.
Still, this improvement does not decrease the degree of $D = V(I_\rbh)$ as a variety (i.e., reduced scheme).
\end{rmrk}

The next examples will use the notation of \Cref{rmrk:improveKemper} and show that the relative boundary hull $D$ produced by Algorithm~\ref{alg:Kemper} can be a highly reducible variety of a very high degree in comparison with $D^\infty$.
We verify the strict inclusion $D^\infty \subsetneq D$.

\begin{exmp}\label{exmp:rationalcurves}
  Consider the rational curves
  \[
    \phi_m: \AA^1_\Q \to \AA^m_\Q, t \mapsto (b_i \coloneqq t^i(t^{2i}+1) \mid i=1, \ldots, m) \mbox{.}
  \]
  and the vanishing ideal $I_m = \langle b_i - t^i(t^{2i}+1) \mid i = 1, \ldots, m \rangle$ of the graph of $\phi_m$.
  It is obvious that the image of $\phi_m$ is closed.
  Let $M_m$ the set of irreducible factors of leading coefficients of $I_m$ in the notation of \Cref{rmrk:improveKemper}.
  Explicit computations show that $|M_m| = m+2$ for $m=2,\ldots,10$ and $( \deg c \mid c \in M_m ) = (2,1,\ldots, 1)$.
  One can verify for $m=1,\ldots,10$ that already $\langle M_m \rangle = \langle c \mid c \in M_m \rangle \subseteq I_\rbh^\infty$ is the unit ideal, hence $D^\infty = V(I_\rbh^\infty) = \emptyset$, and the iteration in Algorithm~\ref{alg:iteration} stops with an empty relative boundary hull, verifying that $\img \phi_m$ is closed.
  In contrast, the range of degrees of the radical ideals of $\langle c, I_B \rangle$ for $c \in M_m$ and the sum of all degrees can be read of the following table:
  \[
  \begin{array}{r|cccc}
    m & 2 & 3 & 4 & 5 \\
    \hline
    \min_{c \in M_m} \deg \sqrt{\langle c, I_B \rangle} & 2 & 3 & 6 & 7 \\
    \max_{c \in M_m} \deg \sqrt{\langle c, I_B \rangle} & 5 & 15 & 15 & 27 \\
    \sum_{c \in M_m} \deg \sqrt{\langle c, I_B \rangle} & 13 & 35 & 55 & 87 \end{array}
  \]
  Starting from $m=6$ at least one of the ideals $\langle c, I_B \rangle$ gets so complicated, that computing its radical with \textsc{Singular} does not terminate in less than an hour.
  This is the reason why the above table stops at $m=5$, since we are not even able to compute all degrees in a sensible time.
\end{exmp}

\begin{exmp} \label{exmp:uMPS224_with_Kemper}
Considering $\operatorname{uMPS}(2,2,4)$ from \Cref{exmp:uMPS} the cardinality of $M$ is $156$ and the sum of degrees of all $c \in M$ is $632$.
The degrees of the radical ideals $\sqrt{\langle c, I_B \rangle}$ for $c \in M$ range between $6$ and $60$ with $3792$ being the sum of all degrees.
\textsc{Singular} failed to compute the minimal associated primes of some of these $156$ ideals.
In stark contrast, the ideal $I_\rbh^\infty$ is a principal prime ideal generated in degree $6$ and one can computationally verify the (very) strict inclusion $D^\infty = V(I_\rbh^\infty) \subsetneq V(I_\rbh) = D$.

Whereas Algorithm~\ref{alg:iteration_graph} (the refined version of Algorithm~\ref{alg:iteration}) in conjunction with Algorithm~\ref{alg:LCA_algebraic} terminates within seconds for $\operatorname{uMPS}(2,2,4)$ (cf.~\Cref{exmp:uMPS}), already the computation of the first relative boundary hull using Algorithm~\ref{alg:Kemper} fails to terminate in less than a day due to the above mentioned degree explosion.
\end{exmp}

Luckily, the smaller example $\operatorname{uMPS}(2,2,3)$ terminates also in conjunction with Algorithm~\ref{alg:Kemper} after a few seconds, but uses much more intermediate steps as shown in the notebook \cite{uMPS223}.
A more impressive jump in the number of intermediate steps when using Algorithm~\ref{alg:Kemper} is demonstrated in the notebook \cite{CFN_reduced}.

\subsubsection{Resultant based methods to generic freeness}

Resultant based methods are conceptually simple, which makes them a prime candidate for a theoretically motivated proof of Chevalley's Theorem, cf.\  \cite[Thm.\ 3.16, Lem.~3.17]{Harris2013algebraic} or \cite[Thm.~3]{chevalley1955schemas} together with \cite[unnumbered lemma]{CartanVarieties} or \cite[\href{https://stacks.math.columbia.edu/tag/00FE}{Theorem 00FE}]{stacks-project} or \cite[Exerc.~II.3.19]{har}.
The case distinctions necessary to compute an image are usually already implemented \cite{MazaMEGA2000,epsilon,regularchains}, in one case even by the second author \cite{thomasalg_jsc,hp_thomasdecomp}.

Resultant based methods replace the projection $\pi:\AA^n_B \twoheadrightarrow \Spec B$ from an $n$-dimensional relative affine space $\AA^n_B$ to a composition of $n$ chosen successive projections $\pi_i:\AA^i_B \twoheadrightarrow \AA^{i-1}_B$, $n\ge i\ge 1$ of a one-dimensional affine spaces.
Each of these $n$ projections $\pi_i$ is described by a univariate polynomial $p_i$ (considered as equation or inequation), which can be computed via resultant methods.
If $p_n$ is an equation, its leading coefficients and its discriminant then yields (a superset of) the non-free locus in $\AA^{n-1}_B$, whereas $\Gamma$ is free over the remaining (generic) subset of $\overline{\pi(\Gamma)}$.
Of course, we need to recursively consider $\AA^{n-1}_B$, which leads to recursively more and more case distinctions.

While Gröbner bases and resultants have comparable degree bounds in the case of two bivariate polynomials \cite{bistritz2010bounds,villard2018computing,lazard1983grobner,lazard1985ideal,buchberger1983note}, resultants have a significantly worse complexity for larger systems.
One reason for this is that resultant based methods emulate the lexicographical term order of Gröbner bases when using the iterated projections, whereas our algorithm is content with any block elimination order.
Furthermore, resultant based methods show behavior close to the worst case bounds, whereas Gröbner bases are often able to exploit sparsity in the inputs.
As a consequence, resultant based methods compute the closure of the image of a morphism significantly slower than Gröbner bases.
Since computing this closure is the bottleneck of our algorithm, resultant based methods often fail to even compute this preliminary step for computing the image of a morphism.
For example, the computations of \Cref{exmp:rationalcurves} for $m=4$ do not terminate after a CPU-hour with
\begin{itemize}[leftmargin=*]
  \item $\mathtt{Triangularize}$ (\textsf{Maple} package $\mathtt{RegularChains}$ \cite{ctd}),
  \item $\mathtt{AlgebraicThomasDecomposition}$ (\textsf{Maple} package $\mathtt{AlgebraicThomas}$ \cite{thomasalg_jsc}),
  \item $\mathtt{PolynomialMapImage}$ (\textsf{Maple} package $\mathtt{RegularChains}$ \cite{ctd}),
  \item $\mathtt{Comprehensive}$ (\textsf{Maple} package $\mathtt{AlgebraicThomas}$ \cite{thomasalg_jsc}).
\end{itemize}
In contrast, Algorithm~\ref{alg:iteration} in conjunction with Algorithm~\ref{alg:LCA} terminates in less than $10$ seconds for $m=15$.

The recursive approach of resultants yields an algorithmic, albeit technically complicated, proof of generic freeness for affine $k$-algebras and suffers from all the drawbacks induced by the successive projections as discussed in \Cref{remark_second_induction}, in addition to choosing an ordered coordinate system.
In particular, this approach tends to induce huge relative boundary hulls.

\subsection{Images via Gröbner covers}

Consider $k$ a \emph{field}.
A Gröbner cover is a stratification of $\Spec B$ into locally closed sets $C_i$, such that $\pi^{-1}(C_i)\cap\Gamma$ can be described by a single Gröbner basis with leading ideal independent of specialization of the $b_j$'s to elements of the algebraic closure $\overline{k}$ \cite{MW10}.
In particular, the coordinate ring of $\pi^{-1}(C_i)\cap\Gamma$ is free over the coordinate ring of $C_i$ for all $i$, hence we have a much stronger condition than non-empty fibers.
The construction of a Gröbner cover is algorithmic, i.e., Gröbner covers are another constructive approach to generic freeness. 
Gröbner covers gives us the image of a projection on a silver platter: the image $\pi(\Gamma)$ is the union of all locally closed sets $C_i$, where the corresponding Gröbner basis is not $\{1\}$.
The computational cost of Gröbner covers is high, as it demands the existence of a uniform Gröbner basis, a much stronger property than only non-zero fibers or generic freeness.
Furthermore, coordinates and term orders need to be chosen.

\subsection{Less effective approaches}

There are additional proofs of Chevalley's Theorem, which are in principle constructive, but have not been suggested as algorithmic.

Some proofs build on the Noether normalization to ensure the existence of preimages outside of a hypersurface of codimension 1 (cf.\ \cite[Thm.~1.8.4, Lem.~1.8.5.1]{EGA4}, \cite[V.\textsection3.1.Cor.~3]{Bou72}, \cite[Lem.~10.18]{goertzwedhorn}).
This does probably not yield an efficient algorithm, as the Noether normalization in itself is usually rather expensive, leads to a computationally costly base change that removes the sparsity in the input, the additional property of a Noether normalization is rather specific, and hence the resulting hypersurfaces is usually big.
The Noether normalization is in some sense dual to our approach: the dimension of the fibers is brought down to zero, not by intersecting with hyperplanes, but by enlarging the base.

An old version of Vakil's lecture notes \cite[8.4.2]{vakilrisingseaold} also constructs the relative boundary hull via the projection of points added at infinity.
However, similar to resultant based methods, the general setting $\pi:\AA_B^n \twoheadrightarrow \AA_B^0 = \Spec B$ is reduced to $n$ successive projections $\pi_i:\AA^i_B \twoheadrightarrow \AA^{i-1}_B$, $n\ge i\ge 1$ and the projective closure is only taken in this reduced case.
However, this approach suffers from the properties described in \Cref{remark_second_induction}.
We expect this succession of projections to unnecessarily enlarge the relative boundary hull.

\subsection{The approach of Harris, Michalek, and Sertöz}\label{subsection_Leipzig}

Our approach is closely related to \cite{ComputingImages}: both make the map more well-behaved (closed resp.\ everywhere defined) by extending their domains (by infinity resp.\  an exceptional divisor of a blowup) and get a relative boundary hull as the image of the extension of the domain.
An additional similarity is that both approaches need generically zero-dimensional fibers of the map.
This necessitates the only sources of choices: the affine subspaces with which we intersect.
The main differences lies in the respective settings: our paper works in an affine setting, whereas \cite{ComputingImages} works in a projective setting.

These two settings are special cases of one another, where the affine setting is reduced to the projective setting (cf.\ \cite[\S2.1]{ComputingImages}\footnote{Beware: the formulas (when interpreted verbatim) only cover the case of \emph{homogeneous polynomial} maps.}) and the projective setting can be reduced to the affine setting via
\begin{itemize}
	\item the stratification $\PP^n=\AA^n\uplus\ldots\uplus\AA^0$,
	\item extending a map
	\begin{align*}
	  f:
	  \begin{cases}
	    \Spec k[x_1,\ldots,x_n] &\to \Spec k[b_1,\ldots,b_m] \mbox{,} \\
	    (x_1,\ldots, x_n) & \mapsto (b_i) \coloneqq (f_i(x_1,\ldots,x_n))
	  \end{cases}
	\end{align*}
	to
	\begin{align*}
	  \widetilde{f}:
	  \begin{cases}
	    \Spec k[x_1,\ldots,x_n,s,t]/\langle st-1\rangle &\to \Spec k[b_1,\ldots,b_m] \mbox{,}\\
	    (x_1, \ldots, x_n, s,t ) & \mapsto (b_i) \coloneqq (s\cdot f_i(x_1,\ldots,x_n))
	  \end{cases}
	\end{align*}
	to include the operation of the one-dimensional torus\footnote{The important special case is $k^*$ for a field $k$.} over $k$, and 
	\item removing the irrelevant locus with vanishing ideal $\langle b_1,\ldots,b_m\rangle$.
\end{itemize}
In our experience, each of these two algorithms works best when applied in their intended setting, as pressing them into the other framework leads to unnecessary inefficiencies.

\begin{exmp} \label{example_plane_line_point}
	Consider the rational map
	\begin{align*}
		f:\AA^2\to\AA^2:(x_1,x_2)\mapsto (b_1,b_2) \coloneqq (x_1,x_1x_2)
	\end{align*}
	with all of $\AA^2$ as closure of the image.
	The fibers are already generically zero-dimensional.
	
	The approach from this paper considers the graph $\{b_1-x_1,b_2-x_1x_2\}\subseteq\AA^2\times\AA^2$.
	We get a relative boundary hull via
	\begin{align*}
		\left(\langle b_1-x_1,b_2-x_1x_2\rangle^{\mathrm{maxdeg}}:\langle x_1,x_2\rangle^\infty\right)_{|x_1=x_2=0}
		&=\left(\langle b_1x_2,x_1\rangle:\langle x_1,x_2\rangle^\infty\right)_{|x_1=x_2=0}\\
		&=\langle x_1,b_1\rangle_{|x_1=x_2=0}\\
		&=\langle b_1\rangle.
	\end{align*}
	
	The approach from \cite{ComputingImages} considers $f$ extended, but retaining the same image, to 
	\begin{align*}
		f':\PP^2\to\PP^2:(x_0:x_1:x_2)\mapsto (x_0^4:x_0^3x_1:x_0^2x_1x_2)\mbox{.}
	\end{align*}
	The blowup $\operatorname{BL}_{\varphi'}\PP^2$ is defined as subset of $\PP^2\times\PP^2$ by the $\Z^2$-homogeneous ideal\footnote{In the computation of the blowup one can avoid eliminating the auxiliary variable $t$ by computing a row-syzygy matrix $s = \left(\begin{smallmatrix} 0 & -x_2 & x_0 \\ -x_1 & x_0 & 0 \end{smallmatrix}\right)$ of $\left(\begin{smallmatrix} x_0^4 \\ x_0^3x_1 \\ x_0^2x_1x_2 \end{smallmatrix}\right)$ and then multiply $s \cdot \left(\begin{smallmatrix} b_0 \\ b_1 \\ b_2 \end{smallmatrix}\right) = \left(\begin{smallmatrix} b_2x_0-b_1x_2,b_1x_0-b_0x_1 \end{smallmatrix}\right)$.}
	\begin{align*}
	\langle b_0-tx_0^4,b_1-tx_0^3x_1,b_2-tx_0^2x_1x_2\rangle\cap k[x_0,x_1,x_2,b_0,b_1,b_2]=\langle b_2x_0-b_1x_2,b_1x_0-b_0x_1\rangle
	\end{align*}
	in $k[x_0,x_1,x_2,b_0,b_1,b_2]$ with degrees $(1,0)$ for the $x_i$'s and $(0,1)$ for the $b_i$'s.
	The construction of the exceptional divisor $\mathcal{E}$ of this blowup takes the homogeneous ideal
	\begin{align*}
	\delta
	&:=\left(\left(\langle 0\rangle+\langle b_2x_0-b_1x_2,b_1x_0-b_0x_1\rangle\right):\langle x_0^4,x_0^3x_1,x_0^2x_1x_2\rangle^\infty\right)+\langle x_0^4,x_0^3x_1,x_0^2x_1x_2\rangle \\
	&=\langle b_2x_0-b_1x_2, b_1x_0-b_0x_1, x_0^4, x_0^3x_1, x_0^2x_1x_2\rangle
	\end{align*}
	and removes irrelevant components
	\begin{align*}
	\varepsilon&:=\left(\delta:\langle x_0,x_1,x_2\rangle^\infty\right):\langle b_0,b_1,b_2\rangle^\infty\\
	&=\langle b_2x_0-b_1x_2,b_1x_0-b_0x_1,x_0^2x_1x_2,x_0^3x_1,b_0x_0^2x_1,x_0^4,b_0^3x_1^2,b_0^3b_1^2x_1,b_0^3b_1^4\rangle
	\end{align*}
	The image of $\mathcal{E}$ in $\PP^1$ under the map induced by $\phi'$ is then given by the ideal
	\begin{align*}
	\varepsilon \cap k[b_0,b_1,b_2] = \langle b_0^3b_1^4\rangle.
	\end{align*}
	This is a relative boundary hull.
	This is basically the same relative boundary hull as $\langle b_1\rangle$ in our approach, just with additional factor $b_0$ to remove points at infinity and powers introduced by the projective modeling.
\end{exmp}

The approach from \cite{ComputingImages} needs polynomial maps, whereas our approach can flexibly switch between maps (the general case) and projections (which do not duplicate the indeterminates of the domain).

\section{Examples}\label{section_examples}

\begin{exmp}[Rabinowitsch trick revisited] \label{exmp:Rabinowitsch}
  Let $B$ be an affine algebra, $J \unlhd B$, and $p \in B$.
  Consider the locally closed set $\Delta \coloneqq \V(J) \setminus \V(p) = \V(J) \setminus \V( \langle J, p \rangle )\subseteq \Spec B$.
  Rabinowitsch described this locally closed set $\Delta$ as the 
  image of the closed set
  \[
    \Gamma \coloneqq \Delta^\mathrm{rab} \coloneqq \V\left( \langle J, t p - 1 \rangle_{B[t]}\right) \subseteq \AA_B^1 = \Spec B[t]
  \]
  under the natural projection $(\pi: \Spec B[t] \twoheadrightarrow \Spec B) = \Spec( B \hookrightarrow B[t] )$.
  We call $\Gamma \coloneqq \Delta^\mathrm{rab}$ the \textbf{Rabinowitsch cover} of $\Delta$.
  
  Let $G$ be a Gröbner basis of $J$ (w.r.t.\  some global monomial order on $B$).
  Then $G_0 \coloneqq G \cup \{t p - 1 \}$ is a Gröbner basis of $I_0 \coloneqq I \coloneqq \langle J, t p - 1 \rangle_{B[t]}$ w.r.t.\  a block elimination order $t\gg B$ on $B[t]$.
  Indeed, since the fibers of $\pi_{|\Gamma}$ are singletons or empty, Algorithm~\ref{alg:LCA_algebraic} verifies that the closure of the projection is given by the ideal $I_B \coloneqq \langle G_0 \cap B \rangle = \langle G \rangle = J$ and the relative boundary hull by the ideal
  \[
    I_\rbh \coloneqq \left( \langle G_0^\mathrm{maxdeg} \rangle : t^\infty \right)_{|t = 0} \\
    = \underbrace{\left( \langle G \rangle : t^\infty \right)_{|t = 0}}_{= \langle G \rangle = J} + \underbrace{\left( \langle tp \rangle : t^\infty \right)_{|t = 0}}_{= \langle p \rangle} = \langle J, p \rangle \mbox{}.
  \]
  Finally, Algorithm~\ref{alg:iteration} states that this is already the entire projection since the ideal $I + \langle I_\rbh \rangle_{B[t]} = \langle J, t p - 1, p \rangle = \langle 1 \rangle$ is the unit ideal.
  
  The same argument applies to the iterated Rabinowitsch trick where the locally closed $\Delta \coloneqq \V(J) \setminus \V(p_1 \cdots p_n) = \V(J) \setminus \V(p_1) \setminus \ldots \setminus \V(p_n)$ is isomorphic to the closed Rabinowitsch cover
  \[
    \Gamma \coloneqq \V(\langle J, t_1 p_1 - 1, \ldots, t_n p_n - 1 \rangle) \subseteq \AA_B^n = \Spec B[t_1, \ldots, t_n]
  \]
  under the natural projection $\pi \coloneqq \Spec( B \hookrightarrow B[t_1, \ldots, t_n])$.
  
  It is well-known that Rabinowitsch trick may also be generalized to general locally closed subsets $\Delta \coloneqq V(J) \setminus V(\langle q_1, \ldots, q_r \rangle)$.
  This set is obviously the projection of the closed Rabinowitsch cover\footnote{This was important in the proof of \Cref{thm:Chevalley_general} in \Cref{proof_Chevalley_general}.}
  \begin{align}\label{exmp:Rabinowitsch:exception}
    \Gamma \coloneqq \Delta^\mathrm{rab} \coloneqq V(\langle J, (t q_1 - 1) \cdots (t q_r - 1) \rangle) = \bigcup_{i=1}^r V(\langle J, t q_i - 1 \rangle) \subseteq \AA_B^1
  \end{align}
  under the projection $\pi \coloneqq \Spec(B \hookrightarrow B[t])$.
  Although the projection $\pi_{|\Gamma}$ will not be injective in general, it will still have $0$-dimensional fibers over its image.
\end{exmp}

\begin{exmp}[Uniform matrix product states]\label{exmp:uMPS}
  For a field $k$ and $D,d,N \in \Z_{>0}$ consider the map
  \[
  T_{D,d,N}:
  \begin{cases} (k^{D \times D})^d &\to \operatorname{Cyc}^N(k^d) \\
  (M_0,\ldots,M_{d-1}) &\mapsto \sum_{0 \leq i_1, \ldots, i_N \leq d-1} \operatorname{tr}(M_{i_1} \cdots M_{i_N})\, e_{i_1} \otimes \cdots \otimes e_{i_N} \mbox{,}
  \end{cases}
  \]
  where $\operatorname{Cyc}^N(k^d) \leq (k^d)^{\otimes N}$ is the subspace of cyclically symmetric tensors.
  A cyclically symmetric tensor which lies in the image $\operatorname{uMPS}(D,d,N)$ of $T_{D,d,N}$ is called a \textbf{uniform $D$-matrix product state}.
  
  Following \Cref{rmrk:dualidea} it is desirable to find a factorization of $T_{D,d,N}$
  \begin{center}
    \begin{tikzpicture}
      \coordinate (r) at (4,0);
      \coordinate (d) at (0,-1);
      
      \node (S) {$(k^{D \times D})^d$};
      \node (T) at ($(r)$) {$\operatorname{Cyc}^N(k^d)$};
      \node (X) at ($0.5*(r)+(d)$) {$Z_{D,d}$};
      
      \draw[-stealth'] (S) -- node[above]{$T_{D,d,N}$} (T);
      \draw[-doublestealth] (S) -- node[below left]{$\pi_{D,d}$} (X);
      \draw[-stealth'] (X) -- node [below right]{$\widetilde{T}_{D,d,N}$} (T);
    \end{tikzpicture}
  \end{center}
  where $\pi_{D,d}$ is surjective with high dimensional fibers.
    
  A candidate for such a space $Z_{D,d}$ is the spectrum of the so-called trace algebra $C_{D,d}$ which is generated by traces of products $\operatorname{tr}(M_{i_1},\ldots,M_{i_\ell})$ and where $M_0,\ldots,M_{d-1}$ are general $d \times d$ matrices over $k$.
  The trace algebra $C_{D,d}$ is finitely generated by invariant-theoretic arguments.
  Sibirskii has among other things showed in \cite{Sibirskii1968} that $C_{2,2}$ is \emph{freely} generated by the five traces $s_i \coloneqq \operatorname{tr}(M_i)$, $s_{ij} \coloneqq \operatorname{tr}(M_i M_j)$ for $i,j \in \{0,1\}$.
  This means that $Z_{2,2} \cong k^5$ and $T_{2,2,N}: (k^{2 \times 2})^2 \to \operatorname{Cyc}^N(k^2)$ can be replaced by $\widetilde{T}_{2,2,N}: k^5 \to \operatorname{Cyc}^N(k^2)$ with generic fiber dimension $0$ for $N \geq 4$.
  For more details and background information see \cite{uMPS}.
  
  \[
    T_{2,2,4}:
    \begin{cases}
      \Q^5 & \to \Q^6 \cong \operatorname{Cyc}^4(\Q^2) \\
      \left(
      \begin{smallmatrix}
        s_0 \\ s_1 \\ s_2 \\ s_3 \\ s_4
      \end{smallmatrix}
      \right)
      =
      \left(
      \begin{smallmatrix}
        s_0 \\ s_1 \\ s_{00} \\ s_{01} \\ s_{11}
      \end{smallmatrix}
      \right)
      & \mapsto
      \left(
      \begin{smallmatrix}
        x_{0000} \\
        x_{0001} \\
        x_{0011} \\
        x_{0111} \\
        x_{1111} \\
        x_{0101}
      \end{smallmatrix}
      \right)
      =
      \left(
      \begin{smallmatrix}
        -\frac{1}{2} s_0^4+s_0^2 s_2+\frac{1}{2} s_2^2 \\
        -\frac{1}{2} s_0^3 s_1+\frac{1}{2} s_0 s_1 s_2+\frac{1}{2} s_0^2 s_3
        +\frac{1}{2} s_2 s_3 \\
        -\frac{1}{2} s_0^2 s_1^2+s_0 s_1 s_3+\frac{1}{2} s_2 s_4 \\
        -\frac{1}{2} s_0 s_1^3+\frac{1}{2} s_1^2 s_3+\frac{1}{2} s_0 s_1 s_4
        +\frac{1}{2} s_3 s_4 \\
        -\frac{1}{2} s_1^4+s_1^2 s_4+\frac{1}{2} s_4^2 \\
        -\frac{1}{2} s_0^2 s_1^2+\frac{1}{2} s_1^2 s_2+\frac{1}{2} s_0^2 s_4
        +s_3^2-\frac{1}{2} s_2 s_4
      \end{smallmatrix}
      \right) \mbox{.}
    \end{cases}
  \]
  We have implemented our Algorithm in the package $\mathtt{ZariskiFrames}$ \cite{ZariskiFrames}, which relies on the package $\mathtt{Locales}$ mentioned in \Cref{rmrk:canonical_decomp}.
  The image $\operatorname{uMPS}(2,2,4)$ of $T_{2,2,4}$ is computed by the command $\mathtt{ConstructibleImage}$ (see the notebook \cite{uMPS224}).
  The result is the (not locally closed but) constructible set
  \[
    \operatorname{uMPS}(2,2,4) = \left( V(f) \setminus V(I_1) \setminus V(I_2) \right) \cup V(J) \mbox{,}
  \]
  where
  {\tiny
  \begin{align*}
    f \coloneqq & 2 x_{0011}^6
    -12 x_{0001} x_{0011}^4 x_{0111}
    +16 x_{0001}^2 x_{0011}^2 x_{0111}^2
    +4 x_{0000} x_{0011}^3 x_{0111}^2
    -8 x_{0000} x_{0001} x_{0011} x_{0111}^3
    \\ &
    +x_{0000}^2 x_{0111}^4
    +4 x_{0001}^2 x_{0011}^3 x_{1111}
    -x_{0000} x_{0011}^4 x_{1111}
    -8 x_{0001}^3 x_{0011} x_{0111} x_{1111}
    +2 x_{0000} x_{0001}^2 x_{0111}^2 x_{1111}
    \\ &
    +x_{0001}^4 x_{1111}^2
    +8 x_{0001} x_{0011}^3 x_{0111} x_{0101}
    -16 x_{0001}^2 x_{0011} x_{0111}^2 x_{0101}
    -4 x_{0000} x_{0011}^2 x_{0111}^2 x_{0101}
    \\ &
    +4 x_{0000} x_{0001} x_{0111}^3 x_{0101}
    -4 x_{0001}^2 x_{0011}^2 x_{1111} x_{0101}
    +4 x_{0001}^3 x_{0111} x_{1111} x_{0101}
    \\ &
    +8 x_{0000} x_{0001} x_{0011} x_{0111} x_{1111} x_{0101}
    -2 x_{0000}^2 x_{0111}^2 x_{1111} x_{0101}
    -2 x_{0000} x_{0001}^2 x_{1111}^2 x_{0101}
    -x_{0011}^4 x_{0101}^2
    \\ &
    +4 x_{0001}^2 x_{0111}^2 x_{0101}^2
    +4 x_{0000} x_{0011} x_{0111}^2 x_{0101}^2
    +4 x_{0001}^2 x_{0011} x_{1111} x_{0101}^2
    -2 x_{0000} x_{0011}^2 x_{1111} x_{0101}^2
    \\ &
    -4 x_{0000} x_{0001} x_{0111} x_{1111} x_{0101}^2
    +x_{0000}^2 x_{1111}^2 x_{0101}^2
    -2 x_{0000} x_{0111}^2 x_{0101}^3
    -2 x_{0001}^2 x_{1111} x_{0101}^3
    +x_{0000} x_{1111} x_{0101}^4 \mbox{,}
    \\
    I_1 \coloneqq &
    \langle x_{0011}
    -x_{0101},
    4 x_{0001} x_{0111}
    -x_{0000} x_{1111}
    -3 x_{0101}^2,
    \quad
    2 x_{0000} x_{0111}^2
    +2 x_{0001}^2 x_{1111}
    -3 x_{0000} x_{1111} x_{0101}
    -x_{0101}^3,
    \\ &
    4 x_{0001}^3 x_{1111}
    +x_{0000}^2 x_{0111} x_{1111}
    -6 x_{0000} x_{0001} x_{1111} x_{0101}
    +3 x_{0000} x_{0111} x_{0101}^2
    -2 x_{0001} x_{0101}^3 \rangle \mbox{,}
    \\
    I_2 \coloneqq &
    \langle 2 x_{0011}^2 x_{1111}
    -4 x_{0001} x_{0111} x_{1111}
    -x_{0000} x_{1111}^2
    -8 x_{0111}^2 x_{0101}
    +8 x_{0011} x_{1111} x_{0101}
    +3 x_{1111} x_{0101}^2,
    \\ &
    2 x_{0111}^3
    -2 x_{0011} x_{0111} x_{1111}
    +x_{0001} x_{1111}^2
    -x_{0111} x_{1111} x_{0101},
    \\ &
    4 x_{0011} x_{0111}^2
    -6 x_{0001} x_{0111} x_{1111}
    -x_{0000} x_{1111}^2
    -14 x_{0111}^2 x_{0101}
    +12 x_{0011} x_{1111} x_{0101}
    +5 x_{1111} x_{0101}^2,
    \\ &
    2 x_{0001} x_{0111}^2
    -2 x_{0001} x_{0011} x_{1111}
    +x_{0000} x_{0111} x_{1111}
    -x_{0001} x_{1111} x_{0101},
    \quad
    x_{0000} x_{0111}^2
    -x_{0001}^2 x_{1111},
    \\ &
    2 x_{0011}^2 x_{0111}
    -4 x_{0001} x_{0011} x_{1111}
    +x_{0000} x_{0111} x_{1111}
    +2 x_{0001} x_{1111} x_{0101}
    -x_{0111} x_{0101}^2,
    \\ &
    16 x_{0001} x_{0011} x_{0111}
    -16 x_{0001}^2 x_{1111}
    -8 x_{0000} x_{0011} x_{1111}
    +2 x_{0011}^2 x_{0101}
    -12 x_{0001} x_{0111} x_{0101}
    +19 x_{0000} x_{1111} x_{0101}
    -x_{0101}^3,
    \\ &
    2 x_{0001}^2 x_{0111}
    -2 x_{0000} x_{0011} x_{0111}
    +x_{0000} x_{0001} x_{1111}
    -x_{0000} x_{0111} x_{0101},
    \\ &
    8 x_{0011}^3
    -16 x_{0001}^2 x_{1111}
    -12 x_{0000} x_{0011} x_{1111}
    +6 x_{0011}^2 x_{0101}
    -4 x_{0001} x_{0111} x_{0101}
    +25 x_{0000} x_{1111} x_{0101}
    -4 x_{0011} x_{0101}^2
    -3 x_{0101}^3,
    \\ &
    2 x_{0001} x_{0011}^2
    -4 x_{0000} x_{0011} x_{0111}
    +x_{0000} x_{0001} x_{1111}
    +2 x_{0000} x_{0111} x_{0101}
    -x_{0001} x_{0101}^2,
    \\ &
    2 x_{0000} x_{0011}^2
    -4 x_{0000} x_{0001} x_{0111}
    -x_{0000}^2 x_{1111}
    -8 x_{0001}^2 x_{0101}
    +8 x_{0000} x_{0011} x_{0101}
    +3 x_{0000} x_{0101}^2,
    \\ &
    4 x_{0001}^2 x_{0011}
    -6 x_{0000} x_{0001} x_{0111}
    -x_{0000}^2 x_{1111}
    -14 x_{0001}^2 x_{0101}
    +12 x_{0000} x_{0011} x_{0101}
    +5 x_{0000} x_{0101}^2,
    \\ &
    2 x_{0001}^3
    -2 x_{0000} x_{0001} x_{0011}
    +x_{0000}^2 x_{0111}
    -x_{0000} x_{0001} x_{0101} \rangle \mbox{,}
    \\
    J \coloneqq &
    \langle x_{0011} -x_{0101},
    x_{0001} x_{1111} -x_{0111} x_{0101},
    x_{0000} x_{1111} -x_{0101}^2,
    x_{0111}^2 -x_{1111} x_{0101},
    \\ &
    x_{0001} x_{0111} -x_{0101}^2,
    x_{0000} x_{0111} -x_{0001} x_{0101},
    x_{0001}^2 -x_{0000} x_{0101} \rangle \mbox{.}
  \end{align*}
  }%
  Our implementation $\mathtt{ConstructibleImage}$ finished in less than $15$ seconds (using \textsc{Singular}'s Gröbner engine in the background \cite{Singular412} on an Intel Xeon E5-2687W v4).
  We stopped
  \begin{itemize}
    \item $\mathtt{totalImage}$ (\textsc{Macaulay2} package $\mathtt{TotalImage}$ \cite{ComputingImages})\footnote{The image was computed in \cite[Appendix B]{uMPS} with the help of an ad hoc representation theoretic argument.}
    \item $\mathtt{grobcov}$ (\textsc{Singular} package $\mathtt{grobcov.lib}$ \cite{MW10})
    \item $\mathtt{PolynomialMapImage}$ (\textsf{Maple} package $\mathtt{RegularChains}$ \cite{ctd})
    \item $\mathtt{Comprehensive}$ (\textsf{Maple} package $\mathtt{AlgebraicThomas}$ \cite{thomasalg_jsc})
  \end{itemize}
  after 24 hours.
  
  The image $\operatorname{uMPS}(2,2,5)$ of $T_{2,2,5}$ is computed by the command $\mathtt{ConstructibleImage}$ in around $6$ minutes (using \textsc{Singular}'s Gröbner engine in the background \cite{Singular412}).
  The result is the (not locally closed but) constructible set
  \[
    \operatorname{uMPS}(2,2,5) = \left( V(\widetilde{I}) \setminus V(\widetilde{I}') \right) \cup V(\widetilde{J}) \mbox{,}
  \]
  where the output is too big to reproduce here (see the notebook \cite{uMPS225}).
\end{exmp}

\section{Algebraic group actions} \label{sec:orbits}

The approach presented in this paper is well-suited to compute orbits of affine algebraic group actions $\alpha: Y\times G\to Y$, where $G$ is an affine algebraic group $G$ and $Y$ an affine variety.
The computation of the $G$-orbit of an element $y\in Y$ is a special case of our setting as the orbit $y G$ is nothing but the image of the \textbf{orbit morphism}
\begin{align*}
	\alpha_y:G\to Y:g\mapsto y g.
\end{align*}

\begin{prop} \label{prop:Grbh}
  A $G$-invariant relative boundary hull of an orbit is the relative boundary of the orbit.
\end{prop}
\begin{proof}
  If a $G$-invariant relative boundary hull contains a point of the orbit then it must contain the entire orbit, contradicting the definition of a relative boundary hull.
\end{proof}

Using the existence of relative boundary hulls from \cref{thm:lca_of_image} and \Cref{prop:Grbh} one can easily reprove the following proposition for $Y$ and $G$ both affine over $\Spec \Z$.
\begin{coro}[{\cite[Proposition in Section I.1.8]{BorelLAG}}] \label{prop:Borel}
  The orbit of an affine algebraic group action is locally closed.
\end{coro}
\begin{proof}
  The orbit $y G = \alpha_y(G)$ is the image of a polynomial map and hence admits a relative boundary hull $D$.
  It follows that $D g$ is also a relative boundary hull for all $g \in G$.
  Hence, $\bigcap_{g \in G} D g$ is a $G$-invariant relative boundary hull of $y G$ and the relative boundary of $y G$ by \Cref{prop:Grbh}.
  It follows that $y G = \overline{y G} \setminus \bigcap_{g \in G} D g$ is locally closed (written in canonical form).
\end{proof}

The proof shows how to avoid the iteration in Algorithm~\ref{alg:iteration}:

\begin{coro} \label{coro_intersect_not_iterate}
  Let $D$ be a relative boundary hull of $yG$.
  Then there exists finitely many group elements $g_1,\ldots, g_\ell$ such that $yG = \overline{yG} \setminus (D \cap D g_1 \cap \ldots \cap D g_\ell)$.
  In particular, the iteration in Algorithm~\ref{alg:LCA} can be avoided, i.e., we only need to apply Algorithm~\ref{alg:LCA} once.
\end{coro}
\begin{proof}
The Noetherianity of $Y$ shows that the intersection defining the relative boundary $\bigcap_{g \in G} D g$ in the previous proof is finite and can be computed by choosing random elements $g_1, \ldots, g_\ell$ until $\alpha_y^{-1}( D \cap D g_1 \cap \ldots \cap D g_\ell)$ is empty.
\end{proof}

We consider an instructive example.

\begin{exmp}[Nilpotent orbit of type $A_1$]\label{example_jordan}
	Consider the special linear group $G=\operatorname{SL}_2$ acting via conjugation on the affine space of $2 \times 2$-matrices $Y = \operatorname{Mat}^{2 \times 2} \cong \AA^4$.
	We are interested in the orbit of $y=\begin{bmatrix} 0 & 1 \\ 0 & 0 \end{bmatrix}$, i.e.\ in the set all matrices similar to this Jordan block (independent of the characteristic).
	Hence, consider the orbit morphism
	\begin{align*}
	\alpha_y: G\to Y: 
	\begin{bmatrix} g_{1,1} & g_{1,2} \\ g_{2,1} & g_{2,2} \end{bmatrix}
	& \mapsto 
	\begin{bmatrix} g_{2,2} & -g_{1,2} \\ -g_{2,1} & g_{1,1} \end{bmatrix}
	\begin{bmatrix} 0 & 1 \\ 0 & 0 \end{bmatrix}
	\begin{bmatrix} g_{1,1} & g_{1,2} \\ g_{2,1} & g_{2,2} \end{bmatrix} = 
	\begin{bmatrix} g_{2,1}g_{2,2} & g_{2,2}^2 \\ -g_{2,1}^2 & -g_{2,1}g_{2,2} \end{bmatrix}.
	\end{align*}
	The graph of the orbit morphism $\alpha_y$ is given by 
	\begin{align*}
	I=\left\langle 
	\begin{matrix}
		& g_{2,1}g_{2,2}-b_{1,1}, 	& g_{2,2}^2-b_{1,2}, \\
		& -g_{2,1}^2-b_{2,1}, 		& -g_{2,1}g_{2,2}-b_{2,2}, \\
		& \multicolumn{2}{c}{g_{1,1}g_{2,2}-g_{1,2}g_{2,1}-1} 
	\end{matrix}
	\right\rangle \unlhd B[g_{1,1},g_{1,2},g_{2,1},g_{2,2}]
	\end{align*}
	for\footnote{$G$ and $Y$ are both defined over $\Z$.} $B=\Z[b_{1,1},b_{1,2},b_{2,1},b_{2,2}]$.
	This yields the closure of the orbit 
	\begin{align*}
		I\cap B=\langle b_{1,1}+b_{2,2}, b_{1,1}b_{2,2}-b_{1,2}b_{2,1}\rangle,
	\end{align*}
	i.e., the matrices with zero trace and zero determinant.
	
	We make the dimension of a generic fiber zero-dimensional by intersecting with $\{g_{1,1}=0\}$.
	This does not change the image closure; $\langle I,g_{1,1}\rangle\cap B=I\cap B$.
	The corresponding relative boundary hull $D_1=\{b_{1,1}=b_{2,1}=b_{2,2}=0\}$ is too big, as $y$ is contained in the orbit $yG$, but also in $D_1$.
	The action of $\begin{bmatrix} 0 & 1 \\ -1 & 0 \end{bmatrix}\in G$ yields a second relative boundary hull $D_2=D_1g=\{b_{1,1}=b_{1,2}=b_{2,2}=0\}$.
	The intersection
	\begin{align*}
		D=D_1\cap D_2=\{b_{1,1}=b_{1,2}=b_{2,1}=b_{2,2}=0\}
	\end{align*}
	is the minimal relative boundary hull, i.e., the relative boundary, and the orbit is locally closed:
	\begin{align*}
		y G = \{ b_{1,1}+b_{2,2} = b_{1,1}b_{2,2}-b_{1,2}b_{2,1} = 0 \}\setminus \{b_{1,1}=b_{1,2}=b_{2,1}=b_{2,2}=0\} \mbox{.}
	\end{align*}
\end{exmp}

For injective orbit morphisms, i.e., for principal orbits we can say more.

\begin{prop}\label{prop:inj_orb_mor}
  If the orbit morphism is injective, then Algorithm~\ref{algorithm_zero_dimensional_fibers_heuristic} which passes from $\Gamma \leadsto \Gamma_0$ can be skipped in Algorithm~\ref{alg:LCA} and the latter yields the relative boundary of the orbit.
  In particular, Algorithm~\ref{alg:iteration} terminates after a single call of Algorithm~\ref{alg:LCA}.
\end{prop}
\begin{proof}
  Since the orbit morphism is injective, the dimension of any non-empty fiber is zero and Algorithm~\ref{algorithm_zero_dimensional_fibers_heuristic} is not needed, i.e., $\Gamma_0 = \Gamma$.
  Now, Algorithm~\ref{alg:LCA} is devoid of any choices and works intrinsically with $G$-invariant inputs.
  Thereby also its output, the relative boundary hull, is $G$-equivariant and therefore is the relative boundary of the orbit by \Cref{prop:Grbh}.
\end{proof}

\subsection{Another heuristic for computing a generic affine subspace}\label{subsect:generic_affine}

The algorithmic approach of this paper can be improved for orbits of irreducible groups: Reduce the dimension of the fiber to zero by intersecting with a complement of the embedded tangent space of the stabilizer at the identity element (=complement of the Lie algebra of the stabilizer).

Recall that orbits of groups are homogeneous and smooth, and the non-empty fibers under
\begin{align*}
	\alpha_y:G\to Y:g\mapsto yg
\end{align*}
for $y\in Y$ are isomorphic \cite[Proposition in Section I.1.8]{BorelLAG}.
Hence, we compute the unique fiber dimension locally via linear algebra at $(y,1_G)\in\Gamma\subseteq Y \times G$, a point in the graph $\Gamma$ of $\alpha_y$.
Then, again via linear algebra, we compute in the tangent space $T_{(y,1_G)}\Gamma$ at $(y,1_G)$ a complement $\widetilde{L}$ to the tangent space $T_{(y,1_G)}F \leq T_{(y,1_G)} \Gamma$ of the fiber at $F \coloneqq \pi^{-1}(y)$.
The tangent space $T_{(y,1_G)}F$ along the fiber is computed by considering only derivatives in the direction of the group $G$.
For the hyperplane $L:=(y,1_G)+\widetilde{L}$, the set $\Gamma_0\coloneqq\Gamma\cap L$ satisfies conditions \ref{condition_0_0}, \ref{condition_0_1}, and \ref{condition_0_2} from \Cref{subsection_algebraic_dimension}.
Interpreting $L\subseteq G$ yields almost a system of representatives of $G_y\backslash G$ (some classes outside an open dense set are met more or less than once).
We formalize this in an algebraic language in Algorithm~\ref{algo:gs}.

\begin{algorithm}[H] \label{algo:gs}
	\SetKwIF{If}{ElseIf}{Else}{if}{then}{elif}{else}{}%
	\DontPrintSemicolon
	\SetKwProg{IntersectWithGenericAffineSubspace}{IntersectWithGenericAffineSubspace}{}{}
	\LinesNotNumbered
	\KwIn{
		The graph $\Gamma \coloneqq \Gamma_{\alpha_y} \subseteq Y\times G$ of the action homomorphism $\alpha_y:G\to Y,\, g\mapsto gy$ of an irreducible algebraic group $G$ on an algebraic variety $Y$.\newline
		We consider $G$ resp.\  $Y$ given by their coordinate rings $k[x_1,\ldots,x_n]/J_G$ resp.\ $k[b_1,\ldots,b_m]/J_Y$, the graph $\Gamma$ by its vanishing ideal 
		\begin{align*}
		I=\langle f_1,\ldots,f_\ell\rangle\unlhd k[b_1,\ldots,b_m,x_1,\ldots,x_n]/(J_G+J_Y),
		\end{align*}
		and the points $1_G\in G$ resp.\ $y\in Y$ by coordinates $(\xi_1,\ldots,\xi_n)\in k^n$ resp.\ $(\beta_1,\ldots,\beta_m)\in k^m$.
	}
	\KwOut{
		a closed affine subset $\Gamma_0\subseteq\Gamma$ given by an ideal 
		\begin{align*}
			I_0\unlhd k[b_1,\ldots,b_m,x_1,\ldots,x_n]/(J_G+J_Y)
		\end{align*}
		such that \Cref{condition_0_0}, \Cref{condition_0_1}, and \Cref{condition_0_2} from \Cref{subsection_algebraic_dimension} are satisfied.
	}
	\IntersectWithGenericAffineSubspace(){($f_1,\ldots,f_\ell,\xi_1,\ldots,\xi_n,\beta_1,\ldots,\beta_m$)}{
		\nl  $J^I_{x;\xi,\beta} \coloneqq \left(\begin{matrix} \frac{\partial (f_a)_{|b_j = \beta_j}}{\partial x_i} \end{matrix}\right)_{|x_i = \xi_i} \in k^{\ell \times n}$ \tcp*{\nameft{Jacobi}an in fiber direction at $1_G$}
		\nl  $\gamma \coloneqq \mathtt{REF}\left(J^I_{x;\xi,\beta}\right)$
			\tcp*{compute the row echelon form}
		\nl  Define $N \subseteq \{ 1, \ldots, n \}$ as the set of column positions of $\gamma$ without pivots\;
		\nl  $E \coloneqq \langle x_i-\xi_i \mid i \in N \rangle$ \tcp*{define the subspace $L \subseteq \AA_B^n$}
		\nl  $I_0 \coloneqq I + E$ \tcp*{intersect $\Gamma \cap L = \Gamma_0=V(I_0)$}\label{algo:gs:I0}
		\nl  \Return{$I_0$}\;
	}
	\caption{IntersectWithGenericAffineSubspace}
\end{algorithm}

Summing up:
\begin{prop} \label{prop_closure_pi_Gamma_0}
  When applying Algorithm~\ref{alg:LCA} to compute the locally closed projection of $\Gamma \coloneqq \Gamma_{\alpha_y}$ we can replace Algorithm~\ref{algorithm_zero_dimensional_fibers_heuristic} by Algorithm~\ref{algo:gs} which only needs derivatives\footnote{Computing $J^I_{x;\xi,\beta} \coloneqq \left(\begin{matrix} \frac{\partial (f_a)_{|b_j = \beta_j}}{\partial x_i} \end{matrix}\right)_{|x_i = \xi_i}$ is more efficient than computing $J^I_{x;\xi,\beta} \coloneqq \left(\begin{matrix} \frac{\partial f_a}{\partial x_i} \end{matrix}\right)_{|x_i = \xi_i,b_j = \beta_j}$.} and the Gaussian algorithm over a field. In particular, the elimination and primary decomposition needed in Algorithm~\ref{algorithm_zero_dimensional_fibers_heuristic} to compute $\overline{\pi(\Gamma)}$ can be avoided and the orbit closure $\overline{Gy}$ can be computed as the image closure of $\overline{\pi(\Gamma_0)} = \overline{\pi(\Gamma)}$ with an elimination involving the usually smaller $\Gamma_0 \subseteq \Gamma$.
\end{prop}

\begin{rmrk}
In principle, this approach of reducing the fiber dimension to zero works for a general $\Gamma$ by considering the tangent space along the fibers of a (smooth) point in $\Gamma$.
However, we face several difficulties not existing in the case of group orbits.
First, unlike the case of group orbits, there is no guarantee that the tangent space along the fiber at the chosen smooth point is of generic fiber dimension, hence one a priori needs to compute the generic fiber dimension, usually via elimination.
Second, in general, we usually do not have a smooth point at hand and even though constructing such an ideal is possible, it is non-trivial and necessary in each step; in case of group orbits we can simply take $(y,1_G)$ and we are done in one step.
Third, the $\xi_i$'s defining the linear subspace $E=\langle x_i-\xi_i \mid i \in N \rangle$ tend to be more complicated than in the group orbit case (there the $\xi_i$'s, representing the identity $1_G$ are usually zero or one).
This problem gets more pronounced in later steps of the algorithm.
\end{rmrk}

\begin{exmp}[Nilpotent cone in type $A_1$, cf.\ \Cref{example_jordan}, revised]\label{example_jordan2}
	Consider again the special linear group $G=SL_2$ acting via conjugation on the affine space of $2 \times 2$-matrices $Y = \operatorname{Mat}_\Z^{2 \times 2} \cong \AA_\Z^4$ and the orbit of $y=\begin{bmatrix} 0 & 1 \\ 0 & 0 \end{bmatrix}$.
	The graph of the action morphism $\alpha_y$ is given by
	\begin{align*}
		I=\left\langle 
		\begin{matrix}
		& g_{2,1}g_{2,2}-b_{1,1}, 	& g_{2,2}^2-b_{1,2}, \\
		& -g_{2,1}^2-b_{2,1}, 		& -g_{2,1}g_{2,2}-b_{2,2}, \\
		& \multicolumn{2}{c}{g_{1,1}g_{2,2}-g_{1,2}g_{2,1}-1} 
		\end{matrix}
		\right\rangle \unlhd B[g_{1,1},g_{1,2},g_{2,1},g_{2,2}]
	\end{align*}
	and we know that $(y,1_G)=(\beta,\xi)=((0,1,0,0);(1,0,0,1))\in\Gamma=V(I)$.
	The Jacobian
	\begin{align*}
		J^I_{x;\xi,\beta} \coloneqq \left(\begin{matrix} \frac{\partial (f_a)_{|b_j = \beta_j}}{\partial x_i} \end{matrix}\right)_{|x_i = \xi_i} =
		\left(\begin{matrix} 
		0 & 0 & 1 & 0 \\
		0 & 0 & 0 & 2 \\
		0 & 0 & 0 & 0 \\
		0 & 0 & -1 & 0\\
		1 & 0 & 0 & 1
		\end{matrix}\right)
	\end{align*}
	suggests $L=\{g_{1,2}=0\}$ .
	This yields a cheaper way to compute the closure of the orbit 
	\begin{align*}
		I\cap B=\langle I,g_{1,2}\rangle\cap B=\langle b_{1,1}+b_{2,2}, b_{1,1}b_{2,2}-b_{1,2}b_{2,1}\rangle
	\end{align*}
	along with a relative boundary hull.
\end{exmp}

\subsection{Epi-mono decomposition} \label{sec:epimono}

Following \Cref{rmrk:dualidea} an ideal approach for computing the orbit would be to use the epi-mono decomposition $G \twoheadrightarrow G_y\backslash G \xhookrightarrow{\iota_y} Y$ of $\alpha_y$, where $G_y$ is the stabilizer of $y \in Y$.
Note that $G_y\backslash G$ is in general quasi-projective and not affine.
Then $y G = \img \alpha_y = \img \iota_y$ and the monic $\iota_y$ has trivially zero dimensional fibers.
This approach has several advantages:
\begin{itemize}
  \item The description of $G_y\backslash G$ as a preparatory step is independent from the space $Y$ which in applications tends to be of much larger dimension than the group $G$.
  \item The fibers are singletons, so the initial step replacing $\Gamma \leadsto \Gamma_0$ is obsolete.
  This removes the arbitrariness in choosing hyperplanes in our algorithm.
  \item Since $\iota_y$ is $G$-equivariant the relative boundary hull will automatically be $G$-equi\-variant and hence will coincide with the relative boundary (and we are done without subsequent ``invariantization'').
\end{itemize}

We will pursue the algebraic compilation of this approach in future work.
The challenge will be to compute the monic $\iota_y$ given $\alpha_y$ without the explicit pre-computation of the rational invariants of the action of $G_y$ on $G$ by multiplication from the left.

\subsection{Finitely many orbits}

Examples \ref{example_jordan} and \ref{example_jordan2} considered a single orbit.
Its closure contained another single orbit (consisting of the zero matrix).
Of course, in general, a group operation can have infinitely many orbits\footnote{e.g., the action of $\operatorname{SL}_2$ via conjugation on $2 \times 2$-matrices over an infinite field.} and even the closure of an orbit can contain infinitely many orbits\footnote{e.g., the right action of $\operatorname{GL}_2$ on $2 \times 2$-matrices over an infinite field yields infinitely many column reduced echelon forms as representatives of orbits.}.
Below we mention some algorithmic benefits applicable when a $G$-space partitions into finitely many orbits.

Assume the action admits finitely many orbits and let $\{y_1,\ldots,y_\ell\}$ be a set of representatives of $G$-orbits.
Then we can compute the closure $\overline{y_i G}$, defined by an ideal $J_i$, of any orbit via elimination.
Note that \Cref{prop_closure_pi_Gamma_0} is applicable for each of these eliminations, with stronger gains in efficiency for smaller orbits.
Determining the containment of the closures of the orbits in one another is an ideal membership test: $\overline{y_iG}\supseteq \overline{y_jG}$ iff $J_i\subseteq J_j$.
Hence, the description of any such orbit $y_iG$ as a locally closed set can be given by the difference
\begin{align*}
	y_iG=\overline{y_iG}\setminus\bigcup_{\overline{y_iG}\supseteq \overline{y_jG}} \overline{y_jG}
\end{align*}
of its closure $\overline{y_iG}$ and the closures of all (maximal) orbits contained in it.
The containment of the closures induces defines the finite stratification of the $G$-space by its orbits.
The nilpotent cone of a semisimple algebraic group $G$ is such space.

Another important class of group operations on varieties with finitely many orbits\footnote{The only connected algebraic groups $G$ guaranteed to always produce finitely many orbits for any operation with a dense orbit are either a torus or a product of a torus and $\textbf{G}_a$, cf.~\cite[Theorem~2]{popov2017algebraic}.} are the torus operations on normal affine toric varieties (cf.\ \cite{CLS11}).
The combinatorial description of toric varieties via the orbit-cone-correspondence is enough to classify orbits.
It also provides a distinguished representative of each orbit.
Furthermore, for the non-maximal orbits there is an explicit description of an epi-mono decomposition of the orbit morphism (cf. \Cref{sec:epimono}).
For the convenience of the reader, we have summarized the relevant results of toric varieties in \Cref{section_primer_toric}.

\begin{exmp}
	We consider the non-smooth cone $\sigma = \operatorname{Cone}(e_1,e_2,e_1+e_3,e_2+e_3) \subset N_\R \equiv \R^3$ with standard lattice $N = \Z e_1 \oplus \Z e_2 \oplus \Z e_3$.
	The dual cone $\sigma^\vee$ is generated by the Hilbert basis $\mathscr{H} = \{ e_1,e_2,e_3,e_1+e_2-e_3 \}$, which defines the orbit morphism
{\small
	\begin{align*}
		\alpha^{\mathscr{H}}_{(1,1,1,1)}:
		\left\{
		\begin{array}{rl}
		\Spec \Z[t_1,t_2,t_3,s] / \langle t_1 t_2 t_3 s - 1 \rangle 
		&
		\to \Spec \Z[b_1,b_2,b_3,b_4], \\
		(t_1, t_2, t_3) &
		\mapsto (b_1,b_2,b_3,b_4) \coloneqq (t_1, t_2, t_3, t_1 t_2 t_3^{-1} \coloneqq t_1^2 t_2^2 s) \mbox{.}
		\end{array}
		\right.
	\end{align*}
}
	Since the orbit morphism $\alpha^\mathscr{H}_{(1,1,1,1)}$ is injective, our algorithm computes the dense torus orbit $O(\{0\})$ in one step (cf.~\Cref{prop:inj_orb_mor})
	\begin{align*}
		O(\{0\}) = \underbrace{V(b_1b_2-b_3b_4)}_{U_\sigma} \setminus( V(b_1,b_4) \cup V(b_2,b_4) \cup V(b_1,b_3) \cup V(b_2,b_3) ) \mbox{,}
	\end{align*}
	which we expect from the theory.
	Indeed, the orbit-cone-correspondence yields for each of the four rays $\rho$ of $\sigma$ the distinguished point $p_\rho$ on the corresponding orbit $O(\rho)$ and the equations of the closure of the orbit.
	\begin{center}
		\begin{tabular}{c|c|c}
			ray generator $u_\rho$ & $p_\rho = (b_1,b_2,b_3,b_4) \in \{0,1\}^4\cap O(\rho)$ & orbit closure $\overline{O(\rho)}$ \\
			\hline
			$e_1$ & $(0,1,1,0)$ & $V(b_1,b_4)$ \\
			$e_2$ & $(1,0,1,0)$ & $V(b_2,b_4)$\\
			$e_1+e_3$ & $(0,1,0,1)$ & $V(b_1,b_3)$\\
			$e_2+e_3$ & $(1,0,0,1)$ & $V(b_2,b_3)$
		\end{tabular}
	\end{center}
	The 3-dimensional dense orbit $O(\{0\})$ and the four 2-dimensional orbits are only two layers of the stratification of the toric variety $\overline{O(\{0\})}$ into orbits.
	The facets defined by two neighboring rays induce four 1-dimensional orbits and the cone $\sigma$ induces a singleton orbit.
	
	The action of the torus on the remaining orbits is given by the orbit morphisms
	{\footnotesize
	\begin{align*}
		\alpha^{\mathscr{H}}_{(a_1,a_2,a_3,a_4)}:
		\left\{
		\begin{array}{rl}
			\Spec \Z[t_1,t_2,t_3,s] / \langle t_1 t_2 t_3 s - 1 \rangle 
			&
			\to \Spec \Z[b_1,b_2,b_3,b_4], \\
			(t_1, t_2, t_3) &
			\mapsto (b_1,b_2,b_3,b_4) \coloneqq (a_1t_1, a_2t_2, a_3t_3, a_4t_1 t_2 t_3^{-1}) \mbox{.}
		\end{array}
		\right.
	\end{align*}
	}
	The kernel of these actions/maps can be determined combinatorial, e.g., the kernel of $\alpha^{\mathscr{H}}_{(0,1,1,0)}$ is generated by $t_1$.
	This yields an epi-mono decomposition of $\alpha^{\mathscr{H}}_{(0,1,1,0)}$ and the two-dimensional quotient torus acts via the monic  $\iota_{(0,1,1,0)}:(t_2, t_3)
	\mapsto (0, t_2, t_3, 0)$.
\end{exmp}

\section*{Acknowledgments}
We would like to thank Florian Heiderich for several discussions on algebraic groups and Sebastian Posur for a careful read and many improvement suggestions.
Our thanks also goes to Sebastian Gutsche and Tom Kuhmichel who contributed to the infrastructure of the software we use.
Last but not least we thank the anonymous referees for their helpful comments and suggestions that contributed to the improvement of this paper.

\appendix

\section{Gröbner basis proof for the dimension reduction}\label{section_exist_planes}

We give a proof, based on Gröbner basis theory, which constructs the affine subspace $L$ necessary to reduce the dimension of fibers (cf.\ \Cref{subsection_algebraic_dimension} and the next proposition).
We assume the coefficients ring $k$ to have effective coset representatives (cf.~\Cref{appendix_computable_ring}).
  As mentioned above, we can always assume $k$ to be \emph{reduced}.
  Furthermore, we can also always assume $k$ to be a \emph{domain}, otherwise it is a product of finitely many domains $k=k_1 \times \cdots \times k_r$ and we can split the projection into $r$ different ones $\Spec k_i \times_{\Spec k} \Gamma \to \Spec (k_i \otimes_k B)$.

\begin{prop} \label{thm:IntersectWithGenericAffineSubspace}
	Let $k$ be an infinite domain, $B = k[b_1, \ldots, b_m]$, and $\Gamma \subseteq \AA_B^n$ be a nonempty \emph{irreducible} affine subset of generic fiber dimension $d$ along $\pi_{|\Gamma}$.
	Let $I = \mathcal{I}(\Gamma) \lhd B[x_1, \ldots, x_n]$ be the vanishing ideal of $\Gamma = V(I) \subseteq \AA^n_B$.
	If $k \cap I = \langle 0 \rangle$, then for each $d' \leq d $ there exists an affine subspace $L \subseteq \AA_B^n$ of dimension $\dim_B L = n - d'$ such that $\Gamma_0 \coloneqq L \cap \Gamma$ satisfies (cf.\ \Cref{subsection_algebraic_dimension})
	\begin{itemize}
		\item the fibers of $\Gamma_0$ along the projection $\pi_{|\Gamma_0}$ are generically $(d - d')$-dimensional;
		\item $\overline{\pi(\Gamma)} = \overline{\pi(\Gamma_0)}$.
	\end{itemize}
	Moreover, $L$ can be chosen constant along the fibers, i.e., there exists an $(n-d')$-dimensional affine subspace $L' \subseteq \AA_k^n$ such that $L \coloneqq L' \times_k \Spec B \subseteq \AA_B^n = \AA_k^n \times_k \Spec B$.
	Finally, the subspace $L'$ can be generically chosen.
\end{prop}
\begin{proof}
	
	Define $K \coloneqq \operatorname{Frac}( B / ( I \cap B ) )$ to be the field of fractions of the domain $B/(I\cap B)$.
	Due to our assumption it is a field extension of $k$.
	Since $k$ has effective coset representatives then so does $B$.
	Hence, one can decide equality of elements in $B/(I\cap B)$ and $K$ is a constructive field, in particular, again with decidable equality of elements.
	
	Consider the extended ideal $I^e = \langle I \rangle \lhd K[x_1, \ldots, x_n]$.
	After possible renaming of the indeterminates we can assume that $x_1, \ldots, x_{d'}$ is an independent set modulo\footnote{Note that $d' = \dim I^e$.} $I^e$.
	Compute a \emph{comprehensive} Gröbner basis $\mathcal{G}$ of $I^e \lhd K[x_1,\ldots, x_{d'}][x_{d'+1}, \ldots, x_n]$ and set $q \coloneqq \prod_{p \in \mathcal{G}} \mathrm{LC}(p) \in K[x_1, \ldots, x_{d'}] \setminus \{ 0 \}$.
	Now choose a specialization $\sigma: K[x_1, \ldots, x_{d'}] \to K, x_i \mapsto a_i$ such that $a_i \in k \subseteq K$ and $\sigma(q) = q(a_1, \ldots, a_{d'}) \neq 0$.
	Since $\mathcal{G}$ is a comprehensive Gröbner basis it follows that the $\sigma$-specialization $\mathcal{G}_\sigma \subseteq K[x_{d'+1}, \ldots, x_n]$ is a Gröbner basis of the $\sigma$-specialization $J$ of $I^e$ to $K[x_{d'+1}, \ldots, x_n]$.
	Since $q(a_1, \ldots, a_{d'}) \neq 0$ the Gröbner basis $\mathcal{G}_\sigma$ has no constant polynomials, in particular, the ideal $J$ is not the unit ideal in $K[x_{d'+1}, \ldots, x_n]$ and $\dim J = d-d'$ due to the independence of $x_1,\ldots, x_{d'}$.
	Without loss of generality we can assume all leading coefficients of $\mathcal{G}_\sigma$ to be $1$.
	The product $q'$ of the denominators of the remaining coefficients is an element of $B \setminus I$.
	It follows that $\mathcal{G}_\sigma$ is a Gröbner basis with normalized leading coefficients over the localization $(B/(I \cap B))_{q'}$.
	Now we are done with $L' \coloneqq V(\langle x_1 - a_1, \ldots, x_{d'} - a_{d'}\rangle) \subseteq \AA_k^n$, since a dense subset of $\overline{\pi}(\Gamma)$ has a preimage in $\Gamma\cap L'$ under $\pi$.
\end{proof}

Otherwise, if $k \cap I \neq \langle 0 \rangle$ then we replace $k$ by $k / (k \cap I)$, $B$ by $B / \langle k \cap I \rangle_B$, and $R$ by $R / \langle k \cap I \rangle_R$.
By the discussion at the beginning of this Section we can again assume $k$ to be a domain.
If $k$ is finite, it is a finite field and we consider a finite field extension $\widetilde{k}/k$ over which the linear space constructed in the proof of \Cref{thm:IntersectWithGenericAffineSubspace} over the algebraic closure $\overline{k}$ can be realized (see \Cref{rmrk_random_hyperplane}).
Geometrically this means that we are considering the image of $\widetilde{\Gamma} \coloneqq \Spec \widetilde{k} \times_{\Spec k} \Gamma$ under the composed projection $\widetilde{\pi}: \AA^n_{\widetilde{B}} \twoheadrightarrow \Spec \widetilde{B} \twoheadrightarrow \Spec B$, where $\widetilde{B} \coloneqq \widetilde{k} \otimes_k B$.
Note that $\widetilde{\pi}(\widetilde{\Gamma}) = \pi(\Gamma)$ and that $\Spec \widetilde{B} \subset \Spec B[t]$ is itself affine over $\Spec B$ with zero dimensional fibers.
This finishes the proof of generic existence of the hyperplanes needed in Algorithm~\ref{algorithm_zero_dimensional_fibers_heuristic} and its variant (Algorithm~\ref{algorithm_zero_dimensional_fibers_heuristic_full}) below.

\section{Get rid of the primary decomposition}\label{section_no_primary}

To bring the dimension of the fibers to zero, in particular to apply \Cref{section_exist_planes} correctly, we need irreducible varieties.
\Cref{subsection_algebraic_dimension} applies the primary decomposition to ensure irreducibility, but only after trying some heuristics.
In this section, we present a stronger heuristic approach that makes the primary decompositions redundant.
Despite theoretical interest, this approach does not seem faster than computing a primary decomposition of the smaller $\Gamma_0$ (cf.\ line~\ref{algorithm_zero_dimensional_fibers_heuristic_decomposition} of Algorithm~\ref{algorithm_zero_dimensional_fibers_heuristic}; note that computing a decomposition of the entire $\Gamma$ is much more expensive).

The stronger heuristic uses a hyperplane, which reduces the dimension of the image to construct a split in total space, as in \Cref{exmp_split_in_total_space}.

\begin{algorithm}[H]
	\SetKwIF{If}{ElseIf}{Else}{if}{then}{elif}{else}{}%
	\DontPrintSemicolon
	\SetKwProg{ZeroDimensionalFibersNoDecomposition}{ZeroDimensionalFibersNoDecomposition}{}{}
	\LinesNotNumbered
	\KwIn{
		A closed subset $\Gamma \subseteq \AA_B^n \xtwoheadrightarrow{\pi} \Spec B$.
		We assume $k$ to be an infinite domain and that the composition $\Gamma \to \Spec B \to \Spec k$ is dominant.
	}
	\KwOut{
		A closed affine subset $\Gamma_0'\subseteq\Gamma$ and a (possibly empty) list of additional closed affine subsets $[\Gamma_1',\ldots,\Gamma_e']$ defining a closed affine subset $\Gamma_0 \coloneqq \Gamma_0'\cup \Gamma_1'\cup\ldots\cup \Gamma_e' \subseteq \Gamma$, such that \ref{condition_0_2} and \ref{condition_0_1} are satisfied; more specifically, the fibers in \ref{condition_0_1} are locally zero-dimensional over the subset $\Gamma_0'$ of $\Gamma_0$.
	}
	\ZeroDimensionalFibersNoDecomposition(){($I$)}{
		\nl  $\Gamma_0' \coloneqq \Gamma$\tcp*{candidate $\Gamma_0$, dimension to be decreased below}
		\nl  $s \coloneqq 1$\tcp*{step counter}
		\nl  $\ell \coloneqq []$\tcp*{list of additional components}
		\tcc{decrease the dimension in the fiber, till it is zero}
		\nl  \While{$\dim(\Gamma_0')-\dim(\pi(\Gamma_0'))>0$}{
			\nl  $E \coloneqq \operatorname{RandomHyperplane}(\AA^n_k)\times_{\Spec k}\Spec B$\tcp*{cf.\ \Cref{rmrk_random_hyperplane}}
			\nl  $\Gamma_0''\coloneqq \Gamma_0'\cap E$\tcp*{intersect with hyperplane}
			\nl  \If(\tcp*[f]{intersection decreases dimension\ldots}){$\dim(\Gamma_0'')<\dim(\Gamma_0')$}{
				\nl  \If(\tcp*[f]{\ldots without reducing the image}) {$\overline{\pi(\Gamma_0')}\subseteq \overline{\pi(\Gamma_0'')}$}{
					\nl $\Gamma_0' \coloneqq \Gamma_0''$\tcp*{found smaller $\Gamma_0'$}
				}
				\nl  \ElseIf(\tcp*[f]{split w.r.t.\ reduced image, avoid early}){$s>n$}{
					\tcc{try splitting $\Gamma_0'$ by splitting the base space}
					\nl  $\Delta\coloneqq \overline{\overline{\pi(\Gamma_0')}\setminus\overline{\pi(\Gamma_0'')}}$\tcp*{``complement'' of the reduced image}
					\nl  \If(\tcp*[f]{if the image reduces nontrivially}){$\overline{\pi(\Gamma_0')}\not\subseteq\Delta$}{
						\nl  $\Gamma_0' \coloneqq \pi^{-1}\left(\overline{\pi\left(\Gamma_0''\right)}\right)\cap\Gamma$\tcp*{continue with one part}
						\nl  $\ell\coloneqq \operatorname{Add}(\ell,\pi^{-1}(\Delta)\cap\Gamma)$\tcp*{store the other half}
					}
					\ElseIf(\tcp*[f]{very expensive, avoid early}){$s>4n$}{
						\nl	$\Gamma_1\coloneqq\pi^{-1}\left(\overline{\pi(\Gamma_0'')}\right)\cap\Gamma$\tcp*{preimage of $\overline{\pi(\Gamma_0'')}$}
						\nl  \If(\tcp*[f]{if this preimage is not the entire $\Gamma$}){$\Gamma\not\subseteq\Gamma_1$}{
							\nl $\Gamma_2\coloneqq\overline{\Gamma\setminus\overline{\Gamma_1}}$\tcp*{``complement'' of the preimage}
							\nl  \If{$\Gamma\not\subseteq\Gamma_2$}{
								$\Gamma_0'\coloneqq\Gamma_1$\tcp*{continue with one part}
								\nl  $\ell\coloneqq \operatorname{Add}(\ell,\Gamma_1)$\tcp*{store the other half}
								
							}
						}
					}
				}
			}
			\nl  $s\coloneqq s + 1$\tcp*{increase number of steps}
		}
		\nl  \Return{$[\Gamma_0',\ell]$}
	}
	\caption{ZeroDimensionalFibersNoDecomposition \label{algorithm_zero_dimensional_fibers_heuristic_full}}
\end{algorithm}

Generically, we expect the following behavior in Algorithm~\ref{algorithm_zero_dimensional_fibers_heuristic_full}:
Intersecting with a generic hyperplane decreases the dimension of each component of $\Gamma$.
The decrease in dimension generically happens without decreasing the image, as long as the dimension of the fiber is positive over a dense set of the image.
Once we can no longer reduce the dimension of the fiber, the image of certain components generically needs to shrink.
In this case, $\Gamma$ generically splits into two components, similar to \Cref{exmp_split_in_total_space} (or, more rarely, but computationally less expensive, similar to \Cref{exmp_split_in_base_space}).

\section{Improving the iteration in Algorithm~\ref{alg:iteration}} \label{sec:improve_alg_iteration}

As mentioned in \Cref{rmrk:decompose_rbh} it is often advantageous to exploit known decompositions of the relative boundary hulls or even to compute decompositions in irreducible components.
As a consequence, Algorithm~\ref{alg:iteration_graph} below will then follow a directed graph structure rather than the linear structure in Algorithm~\ref{alg:iteration}.
To achieve this flow in Algorithm~\ref{alg:iteration_graph} we loop over a bookkeeping bipartite directed graph data structure which we will now define.
The poset $(Z,\succeq)$ in the following definition will specialize to the poset of closed reduced subsets of $\Spec B$ in Algorithm~\ref{alg:iteration}.

\begin{defn}
  A \textbf{(bookkeeping) bipartite directed graph data structure} $c$ consists of a background poset $(Z,\succeq)$ together with
  \begin{itemize}
    \item a mutable finite set $\A(c)$ of positive nodes,
    \item a mutable finite set $\D(c)$ of negative nodes, and
    \item a finite FIFO $\mathcal{P}(c)$ for so-called pre-nodes (with push and pop operations),
  \end{itemize}
  where:
  \begin{itemize}
    \item A \textbf{positive node} $A = (\widetilde{A},\operatorname{parents}(A),\operatorname{children}(A))\in\A(c)$ consists of an underlying object $\widetilde{A} \in Z$ together with two mutable subsets
      \[
        \operatorname{parents}(A),\operatorname{children}(A) \subseteq \D(c)
      \]
      satisfying:
      \begin{align*}
        \widetilde{D} \succeq \widetilde{A} \succeq \widetilde{D'} \quad \forall D \in \operatorname{parents}(A), D' \in \operatorname{children}(A) \mbox{.}
      \end{align*}
      Furthermore, require that the map $\widetilde{\cdot}: \A(c) \to Z, A \mapsto \widetilde{A}$ is injective.
    \item A \textbf{negative node} $D=(\widetilde{D},\operatorname{parents}(D),\operatorname{children}(D))\in \D(c)$ consists of an underlying object $\widetilde{D} \in Z$ together with two mutable subsets
      \[
        \operatorname{parents}(D),\operatorname{children}(D) \subseteq \A(c)
      \]
      satisfying:
      \begin{align*}
        \widetilde{A} \succeq \widetilde{D} \succeq \widetilde{A'} \quad \forall A \in \operatorname{parents}(D), A' \in \operatorname{children}(D) \mbox{.}
      \end{align*}
      Furthermore, require that the map $\widetilde{\cdot}: \D(c) \to Z, D \mapsto \widetilde{D}$ is injective.
    \item A \textbf{pre-node} in $\mathcal{P}(c)$ is a triple $(D, \ell, \Gamma)$, where $D \in \D(c)$, $\ell$ a nonnegative integer (called \textbf{level}), and $\Gamma$ an arbitrary object with no further specification.
    \item Let $A\in\A(c)$ be a positive node and $D\in\D(c)$ a negative node.
    Then $D\in\operatorname{parents}(A)$ iff $A\in\operatorname{children}(D)$ and $A\in\operatorname{parents}(D)$ iff $D\in\operatorname{children}(A)$.
    \item For each negative node $D\in\D(c)$ we have $\operatorname{parents}(D)\not=\emptyset$.
  \end{itemize}
  The bipartite directed graph structure is defined by the parents-children relationship.
\end{defn}

We wrote several algorithms to manipulate $c$.
The following algorithm will be used in the main Algorithm~\ref{alg:iteration_graph} to attach the locally closed output $\widetilde{A} \setminus (\widetilde{D_1} \cup \dots \cup \widetilde{D_a})$ of \textbf{LocallyClosedApproximationOfProjection}($\Gamma \cap \pi^{-1}(\widetilde{D})$) (Algorithm~\ref{alg:LCA}) to a pre-existing negative node $D\in \D(c)$.
The closure $\widetilde{A}$ of the projection will be attached as a positive node (with underlying object $\widetilde{A}$) and the $\widetilde{D_i}$'s as negative nodes (with underlying objects $\widetilde{D_i}$).

\begin{center}
\begin{algorithm}[H]
        \SetKwIF{If}{ElseIf}{Else}{if}{then}{elif}{else}{}%
        \DontPrintSemicolon
        \SetKwProg{Attach}{Attach}{}{}
        \LinesNotNumbered
        \KwIn{
        A bipartite directed graph data structure $c$,
        a negative node $D\in\D(c)$,
        an integer $\ell \geq 0$,
        an object $\widetilde{A} \in Z$,
        objects $\{ \widetilde{D_1}, \ldots, \widetilde{D_a}\} \subset Z$, and an object $\Gamma$
        }
        \KwOut{
        Nothing, a side effect on $c$
        }
        \Attach(){($c$, $D$, $\ell$, $A$, $\{D_1, \ldots, D_a\}$, $\Gamma$)}{
                \nl \uIf(){$\exists B \in \A(c): \widetilde{A} = \widetilde{B}$}{
                \nl $A \coloneqq B$\;
                }
                \nl \Else{
                \nl Add $A \coloneqq (\widetilde{A}, \emptyset, \emptyset)$ to $\A(c)$\;
                }
                \nl Add $D$ to $\operatorname{parents}(A)$\;
                \nl \For{$i = 1, \ldots, \ell$}{
                \nl \uIf(){$\exists D \in \D(c): \widetilde{D_i} = \widetilde{D}$}{
                \nl $D_i \coloneqq D$\;
                \nl Create a pre-node $(D_i, \ell+1, \Gamma)$ and push it to $\mathcal{P}(c)$\;
                }
                \nl \Else{
                \nl Add $D_i \coloneqq (\widetilde{D_i}, \emptyset, \emptyset)$ to $\D(c)$\;
                }
                \nl Add $A$ to $\operatorname{parents}(D_i)$\;
                \nl Add $D_i$ to $\operatorname{children}(A)$\;
                }
        }
\caption{Attach positive node \& corresponding negative nodes to negative node}
\label{alg:Attach}
\end{algorithm}
\end{center}

The following algorithm checks whether the FIFO of pre-nodes of the bipartite directed graph data structure $c$ has been exhausted.
It will be used as the while-loop condition in Algorithm~\ref{alg:iteration_graph}:
\begin{center}
\begin{algorithm}[H]
        \SetKwIF{If}{ElseIf}{Else}{if}{then}{elif}{else}{}%
        \DontPrintSemicolon
        \SetKwProg{IsDone}{IsDone}{}{}
        \LinesNotNumbered
        \KwIn{
        A bipartite directed graph data structure $c$
        }
        \KwOut{
        $\mathtt{true}$ or $\mathtt{false}$
        }
        \IsDone(){($c$)}{
                \nl \Return \texttt{IsEmpty}($\mathcal{P}(c)$)\;
        }
\caption{Check if the FIFO of pre-nodes of $c$ is exhausted}
\label{alg:IsDone}
\end{algorithm}
\end{center}

The next algorithm is used in Algorithm~\ref{alg:iteration_graph} to extract the oldest pre-node in the bipartite directed graph data structure $c$.
\begin{center}
\begin{algorithm}[H]
        \SetKwIF{If}{ElseIf}{Else}{if}{then}{elif}{else}{}%
        \DontPrintSemicolon
        \SetKwProg{Pop}{Pop}{}{}
        \LinesNotNumbered
        \KwIn{
        A bipartite directed graph data structure $c$
        }
        \KwOut{
        A decorated pre-node $(D_i, \ell, \Gamma)$
        }
        \Pop(){($c$)}{
                \nl \Return $\operatorname{Pop}(\mathcal{P}(c))$, i.e., return the oldest element in $\mathcal{P}(c)$ and delete it from $\mathcal{P}(c)$\;
        }
\caption{Get the oldest pre-node of $c$}
\label{alg:Pop}
\end{algorithm}
\end{center}

Algorithm~\ref{alg:iteration_graph} will use the following procedure to check if all pre-nodes in $c$ of a certain level have been processed in order to trigger the squashing of $c$.
\begin{center}
\begin{algorithm}[H]
        \SetKwIF{If}{ElseIf}{Else}{if}{then}{elif}{else}{}%
        \DontPrintSemicolon
        \SetKwProg{MinimalLevelOfPreNodes}{MinimalLevelOfPreNodes}{}{}
        \LinesNotNumbered
        \KwIn{
        A bipartite directed graph data structure $c$
        }
        \KwOut{
        A non-negative integer
        }
        \MinimalLevelOfPreNodes(){($c$)}{
                \nl \Return $\min \{\ell \mid ( D, \ell, \Gamma) \in \mathcal{P}(c) \}$\;
        }
\caption{The minimum level among the pre-nodes of $c$}
\label{alg:MinimalLevelOfPreNodes}
\end{algorithm}
\end{center}

After all pre-nodes of a certain level have been processed the bipartite directed graph data structure $c$ can be squashed as follows:
\begin{itemize}
  \item remove all pairs $(D,A) \in \D(c) \times \A(c)$ from $c$ once $D$ is the only parent of $A$ and $A$ is the only child of $D$ and $A=D$;
  \item reassign the parents-children relationship in $c$;
  \item for all positive nodes $A \in \A(c)$ remove all obsolete children $D \in \operatorname{children}(A)$, i.e., remove all those children of $A$ which are smaller than other children of $A$ with respect to $\succeq$.
\end{itemize}

\begin{center}
\begin{algorithm}[H]
        \SetKwIF{If}{ElseIf}{Else}{if}{then}{elif}{else}{}%
        \DontPrintSemicolon
        \SetKwProg{Squash}{Squash}{}{}
        \LinesNotNumbered
        \KwIn{
        A bipartite directed graph data structure $c$
        }
        \KwOut{
        Nothing, a side effect on the bipartite directed graph data structure $c$
        }
        \Squash(){($c$)}{
                \nl \For{$A \in \A(c)$}{
                \nl \If(){$\operatorname{parents}(A) = \{D\}$ and $\operatorname{children}(D) = \{A\}$ and $\widetilde{A} = \widetilde{D}$}{
                \nl \For{$A' \in \operatorname{parents}(D)$}{
                \nl Add the elements of $\operatorname{children}(A)$ to $\operatorname{children}(A')$\;
                }
                \nl \For{$D' \in \operatorname{children}(A)$}{
                \nl Add the elements of $\operatorname{parents}(D)$ to $\operatorname{parents}(D')$\;
                }
                \nl remove $A,D$ from\footnote{i.e., remove $A$ from $\A(c)$ and from $\operatorname{parents}(D')$ and $\operatorname{children}(D')$ of all $D' \in \D(c)$ and remove $D$ from $\D(c)$ and from $\operatorname{parents}(A')$ and $\operatorname{children}(A')$ of all $A' \in \A(c)$.} $c$\;
                }
                }
                \nl \For{$A \in \A(c)$}{
	                \nl \For{$D \in \operatorname{children}(A)$}{
		                \nl \If(){$\exists D' \in \operatorname{children}(A) \setminus \{D\}$ with $\widetilde{D'} \succeq \widetilde{D}$}{
			                \nl remove $D$ from $\operatorname{children}(A)$\;
			                \nl remove $A$ from $\operatorname{parents}(D)$\;
			                \nl \If(){$\operatorname{parents}(D)=\emptyset$}{
			                	\nl remove $D$ from $\D(c)$
			            	}
		                }
	                }
                }
        }
\caption{Squash the bipartite directed graph data structure $c$}
\label{alg:Squash}
\end{algorithm}
\end{center}
The squash operation can be repeated until $\A(c)$ and $\D(c)$ do not decrease any further.

We graphically demonstrate the first loop of \textbf{squash}, which removed a negative (light) red node and its child, a positive (light) green node.
Parents are drawn left of their children.
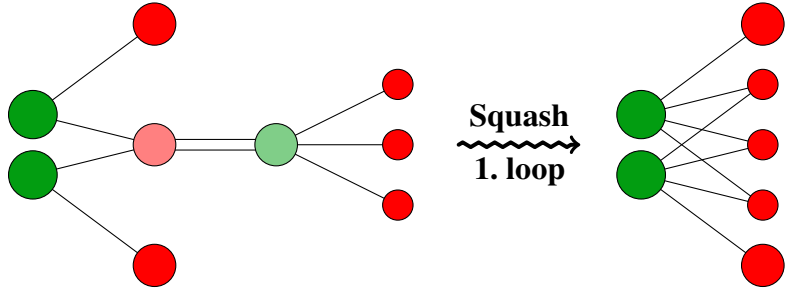
\begin{center}
	\begin{tikzpicture}[scale=.8,cap=round]
		\foreach \x in {0}
			\foreach \y in {2.5,3.5} {
				\filldraw[fill=darkgreen] (\x,\y) circle (0.4);
			}
		\foreach \x in {2}
			\foreach \y in {3} {
				\filldraw[fill=red!50] (\x,\y) circle (0.35);
			}
		\foreach \x in {2}
			\foreach \y in {1,5} {
				\filldraw[fill=red] (\x,\y) circle (0.35);
			}
		\foreach \x in {4}
			\foreach \y in {3} {
				\filldraw[fill=darkgreen!50] (\x,\y) circle (0.35);
			}
		\foreach \x in {6}
			\foreach \y in {2,3,4} {
				\filldraw[fill=red] (\x,\y) circle (0.25);
			}
		\begin{pgfonlayer}{background}
		\foreach \xa / \xb in {0 / 2}
			\foreach \ya in { 2.5, 3.5 }
				\foreach \yb in { 3 }
				{
					\draw[-] (\xa,\ya) -- (\xb,\yb);
				}
		\foreach \xa / \xb in {4 / 6}
			\foreach \ya in { 3 }
				\foreach \yb in { 2, 3, 4 }
				{
				\draw[-] (\xa,\ya) -- (\xb,\yb);
				}
		\draw[-,double,double distance=0.3em] (2,3) -- (4,3);
		\draw[-] (0,2.5) -- (2,1);
		\draw[-] (0,3.5) -- (2,5);
		\end{pgfonlayer}
		
		\draw [->,
		line join=round,
		line width=1.5,
		decorate, decoration={
			zigzag,
			segment length=6,
			amplitude=.9,post=lineto,
			post length=2pt
		}]  (7,3) -- node[above]{\textbf{Squash}} node[below]{\textbf{1.\ loop}} (9,3);
		
		\foreach \x in {10}
			\foreach \y in {2.5,3.5} {
				\filldraw[fill=darkgreen] (\x,\y) circle (0.4);
			}
		\foreach \x in {12}
			\foreach \y in {1,5} {
				\filldraw[fill=red] (\x,\y) circle (0.35);
			}
		\foreach \x in {12}
			\foreach \y in {2,3,4} {
				\filldraw[fill=red] (\x,\y) circle (0.25);
			}
		\begin{pgfonlayer}{background}
		\foreach \xa / \xb in {10 / 12}
			\foreach \ya in { 2.5, 3.5 }
				\foreach \yb in { 2, 3, 4 }
				{
					\draw[-] (\xa,\ya) -- (\xb,\yb);
				}
		\draw[-] (10,2.5) -- (12,1);
		\draw[-] (10,3.5) -- (12,5);
		\end{pgfonlayer}
	\end{tikzpicture}
\end{center}

This diagram shows how the first loop of \textbf{Squash} brings more negative nodes to a single layer, such that the second loop can be more productive in removing negative nodes.
Typically, the negative nodes originally appearing in an earlier generation are $\succeq$-bigger than those appearing in later generations.
Hence, some of these $\succeq$-smaller nodes can be removed by the second loop in \textbf{Squash}, as the next diagram shows.
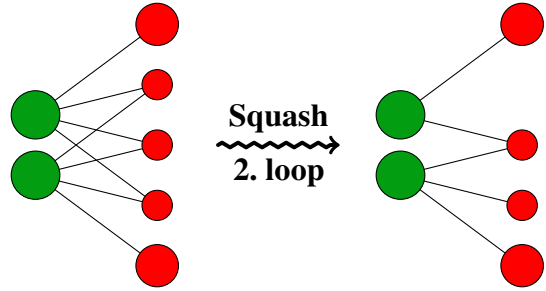
\begin{center}
	\begin{tikzpicture}[scale=.8,cap=round]
	\foreach \x in {0}
		\foreach \y in {2.5,3.5} {
			\filldraw[fill=darkgreen] (\x,\y) circle (0.4);
		}
	\foreach \x in {2}
		\foreach \y in {1,5} {
			\filldraw[fill=red] (\x,\y) circle (0.35);
		}
	\foreach \x in {2}
		\foreach \y in {2,3,4} {
			\filldraw[fill=red] (\x,\y) circle (0.25);
		}
	\begin{pgfonlayer}{background}
	\foreach \xa / \xb in {0 / 2}
		\foreach \ya in { 2.5, 3.5 }
			\foreach \yb in { 2, 3, 4 }
			{
				\draw[-] (\xa,\ya) -- (\xb,\yb);
			}
	\draw[-] (0,2.5) -- (2,1);
	\draw[-] (0,3.5) -- (2,5);
	\end{pgfonlayer}
	\draw [->,
	line join=round,
	line width=1.5,
	decorate, decoration={
		zigzag,
		segment length=6,
		amplitude=.9,post=lineto,
		post length=2pt
	}]  (3,3) -- node[above]{\textbf{Squash}} node[below]{\textbf{2.\ loop}} (5,3);
	
	\foreach \x in {6}
	\foreach \y in {2.5,3.5} {
		\filldraw[fill=darkgreen] (\x,\y) circle (0.4);
	}
	\foreach \x in {8}
		\foreach \y in {1,5} {
			\filldraw[fill=red] (\x,\y) circle (0.35);
		}
	\foreach \x in {8}
	\foreach \y in {2,3} {
		\filldraw[fill=red] (\x,\y) circle (0.25);
		}
	\begin{pgfonlayer}{background} 
	\foreach \xa / \xb in {6 / 8}
		\foreach \ya in { 2.5, 3.5 }
			\foreach \yb in { 3 }
			{
				\draw[-] (\xa,\ya) -- (\xb,\yb);
			}
	\draw[-] (6,2.5) -- (8,1);
	\draw[-] (6,2.5) -- (8,2);
	\draw[-] (6,3.5) -- (8,5);
	\end{pgfonlayer}
	\end{tikzpicture}
\end{center}

The following last helper algorithm is used by Algorithm~\ref{alg:iteration_graph} to convert the bipartite directed graph data structure $c$ into the corresponding constructible set $C$, as a disjoint union of the multiple differences $L = \widetilde{A} \setminus \widetilde{D_1} \setminus \ldots \setminus \widetilde{D_a} = \widetilde{A} \setminus (\widetilde{D_1} \cup \dots \cup \widetilde{D_a})$.
Since we do not assume $(Z, \succeq)$ to be a Boolean algebra we understand multiple differences and their unions as a formal expressions.
In our application they evaluate to elements in the Boolean algebra of constructible sets.

\begin{center}
\begin{algorithm}[H]
        \SetKwIF{If}{ElseIf}{Else}{if}{then}{elif}{else}{}%
        \DontPrintSemicolon
        \SetKwProg{AsUnionOfMultipleDifferences}{AsUnionOfMultipleDifferences}{}{}
        \LinesNotNumbered
        \KwIn{
        A bipartite directed graph data structure $c$
        }
        \KwOut{
        A constructible set $C$ as disjoint union of multiple differences
        }
        \AsUnionOfMultipleDifferences(){($c$)}{
                \nl $C \coloneqq \emptyset$\;
                \nl \For{$A \in \A(c)$}{
                \nl $L \coloneqq \widetilde{A}$\;
                \nl \For{$D \in \operatorname{children}(A)$}{
                \nl $L \coloneqq L \setminus \widetilde{D}$\;
                }
                \nl $C \coloneqq C\, \uplus\, L$\ \tcp*{$L = \widetilde{A} \setminus \widetilde{D_1} \setminus \ldots \setminus \widetilde{D_a}$}
                }
                \Return $C$\;
        }
\caption{The constructible set defined by $c$}
\label{alg:AsUnionOfMultipleDifferences}
\end{algorithm}
\end{center}

We can now formulate the promised substitute of Algorithm~\ref{alg:iteration}:
\begin{center}
\begin{algorithm}[H]
        \SetKwIF{If}{ElseIf}{Else}{if}{then}{elif}{else}{}%
        \DontPrintSemicolon
        \SetKwProg{ConstructibleProjection}{ConstructibleProjection}{}{}
        \LinesNotNumbered
        \KwIn{
        A closed subset $\Gamma \subseteq \AA_B^n \xtwoheadrightarrow{\pi} \Spec B$
        }
        \KwOut{
        $\pi(\Gamma) = C$ as a finite \emph{disjoint} union of locally closed subsets of $\Spec B$
        }
        \ConstructibleProjection(){($\Gamma$)}{
                \nl Initialize a bipartite directed graph data structure $c$ with a single pre-node $(\Spec B, 0, \Gamma)$ for the background poset $(Z, \succeq)$ being the poset of closed reduced subsets of $\Spec B$\;
                \nl \While(\tcp*[f]{as long as there are pre-nodes to process}){not {\bf IsDone}$(c)$}{
                \nl $(D, \ell, \Gamma) \coloneqq \textbf{Pop}(c)$ \tcp*{get the oldest pre-node}
                \nl $\Gamma \coloneqq \Gamma \cap \pi^{-1}(\widetilde{D})$ \tcp*{compute the preimage of $\widetilde{D}$ in $\Gamma$} \label{alg:iteration:preimage}
                \nl $\widetilde{A}, \{\widetilde{D_1}, \ldots, \widetilde{D_a}\} \coloneqq \textbf{LocallyClosedApproximationOfProjection}(\Gamma)$\;\label{alg:iteration_graph_line_lca}
                \nl \If(){$\widetilde{A} \neq \emptyset$}{
                  \nl \textbf{Attach}($c,D,\ell,\widetilde{A},\{\widetilde{D_1}, \ldots, \widetilde{D_a}\}, \Gamma$)\;
                }
                \nl \If(\tcp*[f]{once level $\ell$ is exhausted}){{\bf MinimalLevelOfPreNodes}($c$) $> \ell$}{
                  \nl \textbf{Squash}($c$)\;
                }
                }
                \nl \Return \textbf{AsUnionOfMultipleDifferences}($c$)\;
        }
\caption{Compute constructible projection (geometric)}
\label{alg:iteration_graph}
\end{algorithm}
\end{center}

Note that \textbf{LocallyClosedApproximationOfProjection} might additionally return multiple components $[\Gamma_1',\ldots,\Gamma_e']$ (cf.\ \Cref{remark_additional_components}), which we then add immediately after line \ref{alg:iteration_graph_line_lca} at the \emph{beginning} of $\mathcal{P}(c)$ as new pre-nodes $(D,\ell,\Gamma_i')$.

\begin{rmrk}
  We now compare our use of the bipartite directed graph data structure $c$ with the constructible tree used in \cite{ComputingImages}.
  Whereas they clean up the constructible tree ``at the end of the loop'' to get rid of obsolete nodes, we squash the data structure after each level has been completed in order to avoid computing at least some of the obsolete nodes.
  Otherwise, the first and second loop in the squash Algorithm~\ref{alg:Squash} correspond to the first and second cleaning operation in \cite[Section 3.1.1]{ComputingImages}, respectively.
  A further simplification is that we do not need to delete negative nodes ``and all of its descendants'' since the negative nodes we delete in the second loop will not have any descendants when the level in which they are computed is completed.
\end{rmrk}

\begin{rmrk}
  Algorithm~\ref{alg:iteration_graph} can be parallelized within a single level $\ell$.
  The procedure \textbf{Attach} is a blocking operation and the procedure \textbf{Squash} should be performed in single-threaded mode.
  In particular, the expensive procedure \textbf{LocallyClosedApproximationOfProjection}($\Gamma$) can be called for multiple $\Gamma$'s within a fixed level $\ell$.
\end{rmrk}

\section{A primer on toric varieties}\label{section_primer_toric}

Let $N \cong \Z^n$ be a lattice and $T_N \coloneqq N \otimes_\Z k^* \cong \mathbb{T}_k^n$ the corresponding torus over the algebraically closed field $k$ of characteristic $0$ (with $N$ being the lattice of one-parameter subgroups of $T_N$).
Further let $\sigma \subset N_\R$ be a strongly convex rational polyhedral cone, $\sigma^\vee \subset M_\R$ its dual cone, where $M \coloneqq N^\vee \coloneqq \Hom_\Z(N,\Z)$ is the character lattice of $T_N$.
A Hilbert basis $\mathscr{H} \coloneqq \{m_1, \ldots, m_s\} \subset M$ of the associated saturated affine semigroup $\mathsf{S}_\sigma \coloneqq \sigma^\vee \cap M$ describes the action morphism
\begin{align*}
\alpha^{\mathscr{H}}:
\begin{cases}
\AA^s_k \times T_N & \hookrightarrow \AA^s_k, \\
(\underline{a}, \underline{t}) = (a_1, \ldots, a_s, t_1, \ldots, t_n) & \mapsto (b_1, \ldots, b_s) \coloneqq (a_1 \underline{t}^{m_1}, \ldots, a_s \underline{t}^{m_s}) \mbox{.}
\end{cases}
\end{align*}
of the torus $T_N$ on $\AA^s_k$.
The principal orbit $O(\{0\}) \coloneqq p_0 \cdot T_N$ of the distinguished point $p_0 \coloneqq (1,\ldots,1) \in \AA^s_k$ under the torus action can be described by the image of
\begin{align*}
\alpha^{\mathscr{H}}_{(1,\ldots,1)}:
\begin{cases}
\Spec k[t_1^\pm, \ldots, t_n^\pm] \eqqcolon \mathbb{T}_k^n \cong T_N & \hookrightarrow \AA^s_k \coloneqq \Spec k[b_1, \ldots, b_s], \\
\underline{t} = (t_1, \ldots, t_n) & \mapsto (b_1, \ldots, b_s) \coloneqq (\underline{t}^{m_1}, \ldots, \underline{t}^{m_s}) \mbox{.}
\end{cases}
\end{align*}
The normal affine toric variety $U_\sigma \cong \Spec k[\mathsf{S}_\sigma]$ is defined as the orbit closure $V(\{0\}) = \overline{O(\{0\})}$ in $\AA^s_k$.
The principal orbit $O(\{0\})$ corresponds to the unique $0$-dimensional facet $\{0\}$ of the strongly convex $\sigma$.
More generally, the orbit-cone-correspondence states that each $d$-dimensional facet $\tau \prec \sigma$ corresponds to a distinguished point $p_\tau$ and its $(n-d)$-dimensional orbit $O(\tau) \coloneqq p_\tau \cdot T_N$ with orbit closure $V(\tau) \coloneqq \overline{O(\tau)}$ in $\AA^s_k$.
The coordinates of $p_\tau$ are either $0$ or $1$.
The $i$-th coordinate of $p_\tau$ is $1$ iff $\tau \subset N_\R$ is perpendicular to $m_i \in M_\R$.

The orbit $O(\tau)$ is a principal homogeneous space for the factor group $T_{N(\tau)}\coloneqq N(\tau) \otimes_\Z k^*$ of the full torus $T_N$, for the lattice $N(\tau)\coloneqq N/\Z(\tau\cap N)$.
In particular, each of these orbits admits an epi-mono factorization of the orbit morphism via combinatorial means.

Furthermore, $V(\tau) = \bigcup_{\nu \succeq \tau} O(\nu)$, inducing a stratification of the orbits.
In particular, the locally closed principal torus orbit $O(\{0\})$ admits the description
\[
O(\{0\}) = \underbrace{V(\{0\})}_{\eqqcolon U_\sigma} \setminus \bigcup_{\rho \in \operatorname{Rays}(\sigma)} V(\rho) \mbox{.}
\]
In other words, the relative boundary of the locally closed orbit is $\underline{\partial} O(\{0\}) = \bigcup_{\rho \in \operatorname{Rays}(\sigma)} V(\rho)$.

\newcommand{\etalchar}[1]{$^{#1}$}
\def\cprime{$'$} \def\cprime{$'$} \def\cprime{$'$} \def\cprime{$'$}
  \def\cprime{$'$}
\providecommand{\bysame}{\leavevmode\hbox to3em{\hrulefill}\thinspace}
\providecommand{\MR}{\relax\ifhmode\unskip\space\fi MR }
\providecommand{\MRhref}[2]{%
  \href{http://www.ams.org/mathscinet-getitem?mr=#1}{#2}
}
\providecommand{\href}[2]{#2}

\end{document}